\documentclass{amsart}
\usepackage{amsmath, amsthm, amssymb}
\usepackage[mathscr]{eucal}
\usepackage{hyperref}
\usepackage{enumerate}
\usepackage[all, cmtip]{xy}
\usepackage{enumitem}
\usepackage{tikz}
\usetikzlibrary{arrows, chains, matrix, positioning, scopes}
\allowdisplaybreaks

\newcommand{\script}[1]{{\mathcal{#1}}}
\newcommand{\A}{\script{A}}
\newcommand{\B}{\script{B}}
\newcommand{\Cgh}{\script{C}}
\newcommand{\I}{\script{I}}

\newcommand{\Y}{\script{Y}}

\newcommand{\Zme}{\mathcal{Z}}

\newcommand{\Hil}{\mathcal{H}}

\newcommand{\ip}[2]{\left( \left. #1 \, \right| \, #2 \right)}
\newcommand{\hip}[2]{\langle \! \langle #1, #2 \rangle \! \rangle}

\newcommand{\abs}[1]{{\left| #1 \right|}}
\newcommand{\norm}[1]{{\left\| #1 \right\|}}

\newcommand{\supp}{\operatorname{supp}}
\newcommand{\Prim}{\operatorname{Prim}}

\newcommand{\id}{{\rm{id}}}
\newcommand{\spa}{\operatorname{span}}
\newcommand{\image}{\operatorname{im}}

\newcommand{\omax}{\otimes_{\rm{max}}}

\newcommand{\go}{{G^{\scriptscriptstyle{(0)}}}}
\newcommand{\gtwo}{G^{\scriptscriptstyle{(2)}}}
\newcommand{\ho}{{H^{\scriptscriptstyle{(0)}}}}

\newcommand{\lo}{L^{\scriptscriptstyle{(0)}}}

\newcommand{\iso}{\operatorname{Iso}}
\DeclareMathOperator{\Ind}{Ind}
\newcommand{\Zop}{Z^{\text{op}}}

\newcommand{\lt}{{\rm{lt}}}

\newtheorem{prop}{Proposition}[section]
\newtheorem{thm}[prop]{Theorem}
\newtheorem{cor}[prop]{Corollary}
\newtheorem{lem}[prop]{Lemma}
\theoremstyle{definition}

\newtheorem{exmp}[prop]{Example}
\newtheorem{rem}[prop]{Remark}
\newtheorem{ques}{Question}

\newlist{thmenum}{enumerate}{10}
\setlist[thmenum,1]{label=\textnormal{(\alph*)}}
\setlist[thmenum,2]{label=\textnormal{(\roman*)}}

\newlist{altenum}{enumerate}{10}
\setlist[altenum,1]{label=\textnormal{(\roman*)}}
\setlist[altenum,2]{label=\textnormal{(\alph*)}}

   \definecolor{strawberry}{RGB}{0, 128, 255}

\title{On some Permanence Properties of Exact Groupoids}
\author{Scott M. LaLonde}
\date{\today}
\address{Department of Mathematics, The University of Texas at Tyler, 3900 University Boulevard, Tyler, TX 75799}
\email{slalonde@uttyler.edu}

\begin{document}

\begin{abstract}
	A locally compact groupoid is said to be exact if its associated reduced crossed product functor is exact. In this paper, we establish some permanence properties
	of exactness, including generalizations of some known results for exact groups. Our primary goal is to show that exactness descends to certain types of closed 
	subgroupoids, which in turn gives conditions under which the isotropy groups of an exact groupoid are guaranteed to be exact. As an initial step toward these results, 
	we establish the exactness of any transformation groupoid associated to an action of an exact groupoid on a locally compact Hausdorff space. We also obtain a partial 
	converse to this result, which generalizes a theorem of Kirchberg and Wassermann. We end with some comments on the weak form of exactness known as inner 
	exactness.
\end{abstract}

\maketitle

\section{Introduction}
The notion of an exact group was first introduced by Kirchberg and Wassermann in \cite{kw99} for the purpose of studying the continuity of $C^*$-bundles associated to 
reduced crossed products. They defined a locally compact group $G$ to be \emph{exact} if given any $C^*$-dynamical system $(A, G, \alpha)$ and any $G$-invariant ideal
$I \subseteq A$, the sequence
\[
	0 \to I \rtimes_{\alpha_I, r} G \to A \rtimes_{\alpha, r} G \to A/I \rtimes_{\alpha^I, r} G \to 0
\]
is exact. There are now several other characterizations of exactness, including amenability at infinity \cite{claire02, bcl} and, in the case of discrete groups, the exactness of the 
reduced $C^*$-algebra $C_r^*(G)$ \cite[Theorem 5.2]{kw99}.

Given the level of attention received by groupoids and their $C^*$-algebras in recent years, it is natural and worthwhile to study exactness for locally compact groupoids. Indeed, 
we can generalize the definition given by Kirchberg and Wassermann, and declare a locally compact Hausdorff groupoid to be \emph{exact} if the sequence
\[
	0 \to \I \rtimes_{\alpha_I, r} G \to \A \rtimes_{\alpha, r} G \to \A/\I \rtimes_{\alpha^I, r} G \to 0
\]
is exact for any dynamical system $(\A, G, \alpha)$ and any $G$-invariant ideal $\I$. The study of exact groupoids is fairly new, though they have been investigated in some detail in
\cite{lalondeexact, lalonde2014}, and more recently in \cite{claireexact}.

It is well-known that the full crossed product functor associated to any group or groupoid is exact. At the time \cite{kw99} was written, it was unknown whether there were 
groups for which the exactness of the reduced crossed product failed. However, Gromov later produced examples of non-exact discrete groups. More recently, Higson,
Lafforgue, and Skandalis \cite{hls} have constructed reasonably simple examples of non-exact groupoids. These examples are particularly important, since they also provide
counterexamples to the Baum-Connes conjecture for groupoids.

The aim of this paper is to investigate some permanence properties for exact groupoids. One such property was already observed in \cite{lalondeexact}, where it was shown
that exactness is preserved under equivalence of groupoids. Most of the results that we take up in this paper are inspired by ones established for groups in \cite{kw99-2}. We 
are focused primarily on showing that exactness descends to certain kinds of closed subgroupoids. As a special case, such results have implications for the isotropy bundle and
isotropy groups of an exact groupoid. A crucial intermediate step is the result that if $G$ is an exact groupoid and $X$ is any $G$-space, then the transformation groupoid 
$G \ltimes X$ is exact. We also obtain a partial converse to this result---the existence of a certain exact transformation groupoid guarantees the exactness of $G$. This 
generalizes a result for groups from \cite[\S 7]{kw99-2}.

The structure of this paper is as follows. In Section \ref{sec:background}, we present some background information on locally compact groupoids, crossed products, and
exactness. Section \ref{sec:transformation} is where we establish our first results about exactness for transformation groupoids, and in Section \ref{sec:subgroupoids} we
apply these results to determine when exactness descends to subgroupoids. In Section \ref{sec:converse} we establish a partial converse to the main result of Section
\ref{sec:transformation}. Finally, in Section \ref{sec:inner} we present a brief result together with some examples, comments, and open questions pertaining to the related notion 
of \emph{inner exactness} (as defined by Anantharaman-Delaroche in \cite{claireweak}).

\section{Preliminaries on Exact Groupoids}
\label{sec:background}
Let $G$ be a locally compact Hausdorff groupoid. We denote the unit space of $G$ by $\go$ and the set of composable pairs by $\gtwo$. The range and source maps are denoted 
by $r, s : G \to \go$, respectively. Unless otherwise specified, it is assumed that all groupoids are locally compact, Hausdorff, and second countable, and that they admit continuous 
Haar systems.

Let us introduce a few bits of useful notation. Given a groupoid $G$ and a set $A \subseteq \go$, we define
\[
	G_A = s^{-1}(A), \quad G^A = r^{-1}(A).
\]
The set $G \vert_A = G_A \cap G^A$ is a groupoid, called the \emph{reduction} of $G$ to $A$. Notice that if $A$ is invariant in $\go$ (meaning it is invariant under the natural action 
of $G$ on $\go$), then $G \vert_A = G_A = G^A$. Of particular import is the special case where $A$ is a singleton: if $u \in \go$,
\[
	G \vert_{\{u\}} = \{\gamma \in G : r(\gamma) = s(\gamma) = u\}
\]
is a group, called the \emph{isotropy group} at $u$. The union of all the isotropy groups is called the \emph{isotropy bundle} (or \emph{isotropy subgroupoid}), denoted by
$\iso(G)$.

Much of what we aim to do involves groupoid actions on fibered spaces. Let $X$ be a topological space. We say that $X$ is a \emph{(left) $G$-space} if there is a continuous,
open surjection $p : X \to \go$ and a continuous map $G {_s}*_p X \to X$, denoted by $(\gamma, x) \mapsto \gamma \cdot x$, satisfying
\begin{enumerate}
	\item $p(x) \cdot x = x$ for all $x \in X$
	\item $\gamma \cdot (\eta \cdot x) = \gamma \eta \cdot x$ whenever $(\gamma, \eta) \in \gtwo$ and $s(\eta) = p(x)$.
\end{enumerate}

\begin{exmp}[Transformation groupoids]
	Suppose $X$ is a locally compact Hausdorff space, and $G$ is a groupoid acting on the left of $X$. Define
	\[
		G \ltimes X = \left\{ (\gamma, x) \in G \times X : r(\gamma) = p(x) \right\}.
	\]
	Then $G \ltimes X$ is a locally compact Hausdorff space, which is second countable whenever $G$ and $X$ are. It becomes a groupoid under the operations
	\[
		(\gamma, x)(\eta, \eta \cdot x) = (\gamma \eta, x), \quad (\gamma, x)^{-1} = (\gamma^{-1}, \gamma^{-1} \cdot x).
	\]
	Notice that the range and source maps take the form
	\[
		r(\gamma, x) = (r(\gamma), x), \quad s(\gamma, x) = (s(\gamma), \gamma^{-1} \cdot x),
	\]
	so we identify $(G \ltimes X)^{\scriptscriptstyle{(0)}}$ with $X$ and write $r(\gamma, x) = x$ and $s(\gamma, x) = \gamma^{-1} \cdot x$. (See 
	\cite[Proposition 2.3.11]{lalondethesis} for a more detailed discussion.) Furthermore, if $G$ has a Haar system $\{\lambda^u\}_{u \in \go}$, then 
	$\{\lambda^{p(x)} \times \delta_x\}_{x \in X}$ is a Haar system for $G \ltimes X$, where $\delta_x$ denotes the Dirac measure on $X$ concentrated at $x$ 
	\cite[Proposition 2.3.12]{lalondethesis}.
	
	It is worth noting that the definition of the transformation groupoid is occasionally taken to be
	\[
		G*X = \left\{ (\gamma, x) \in G \times X : s(\gamma) = p(x) \right\}
	\]
	equipped with the appropriate operations. In particular, this is the definition used by Muhly \cite[Remark 2.14(1)]{muhly}.
\end{exmp}

Let $G$ be a groupoid with Haar system $\{\lambda^u\}_{u \in \go}$, suppose $\A$ is an upper semicontinuous $C^*$-bundle over $\go$, and let $A = \Gamma_0(\go, \A)$ denote 
the associated algebra of continuous sections that vanish at infinity. Then $A$ is a $C_0(\go)$-algebra, and the fiber of $\A$ over a point $u$ is the quotient of $A$ by the 
ideal of sections vanishing at $u$. We usually denote this fiber by $\A_u$, though we occasionally write $A(u)$ when it is more helpful (or notationally convenient) to think of $\A_u$ 
in its r\^{o}le as a quotient of $A$. Now suppose $G$ acts on $\A$ via fiberwise isomorphisms---that is, there is a family $\alpha = \{ \alpha_\gamma \}_{\gamma \in G}$, where 
$\alpha_\gamma : \A_{s(\gamma)} \to \A_{r(\gamma)}$ is an isomorphism for each $\gamma \in G$, and the map $(\gamma, a) \mapsto \alpha_\gamma(a)$ defines a continuous
action of $G$ on $\A$. Then the triple $(\A, G, \alpha)$ is called a \emph{groupoid dynamical system}. We say the dynamical system is \emph{separable} if $A$ is separable and 
$G$ is second countable. All dynamical systems in this paper are assumed to be separable.

\begin{exmp}
	Suppose $G$ acts on the left of a locally compact Hausdorff space $X$. Then there is a continuous open surjection $p : X \to \go$, so $C_0(X)$ becomes a 
	$C_0(\go)$-algebra with fibers
	\[
		C_0(X)(u) = C_0(p^{-1}(u))
	\]
	by \cite[Example C.4]{TFB2}. Let $\Cgh$ denote the associated upper semicontinuous $C^*$-bundle. Then by \cite[Proposition 4.38]{geoff} or \cite[Example 4.8]{mw08}, 
	$G$ acts on $\Cgh$ via the family $\lt = \{\lt_\gamma\}_{\gamma \in G}$, where $\lt_\gamma : C_0(p^{-1}(s(\gamma))) \to C_0(p^{-1}(r(\gamma)))$ is defined by
	\[
		\lt(f)(x) = f(\gamma^{-1} \cdot x).
	\]
	Thus $(\Cgh, G, \lt)$ is a groupoid dynamical system, which is separable whenever $G$ and $X$ are second countable.
\end{exmp}

Given a groupoid dynamical system $(\A, G, \alpha)$, the space $\Gamma_c(G, r^*\A)$ of continuous compactly-supported sections becomes a $*$-algebra under the convolution 
product
\[
	f * g(\gamma) = \int_G f(\eta) \alpha_\eta \bigl( g(\eta^{-1} \gamma) \bigr) \, d\lambda^{r(\gamma)}(\eta)
\]
and involution
\[
	f^*(\gamma) = \alpha_\gamma \bigl( f(\gamma^{-1})^* \bigr).
\]
Moreover, $\Gamma_c(G, r^*\A)$ forms a topological $*$-algebra under the \emph{inductive limit topology}. (Recall that a net converges in the inductive limit topology if it converges 
uniformly and the supports of its elements are eventually contained in a fixed compact set.) However, there are many different ways to define a $C^*$-norm on $\Gamma_c(G, r^*\A)$. 
One such example is the \emph{universal norm}, which we define by setting $\norm{f} = \sup \norm{\pi(f)}$, where $\pi$ ranges over all representations of $\Gamma_c(G, r^*\A)$ on 
Hilbert space that are continuous in the inductive limit topology. The resulting completion is the \emph{full crossed product} of $\A$ by $G$, denoted by $\A \rtimes_\alpha G$. 

There is also a notion of reduced groupoid crossed product. As with other reduced constructions, its norm is easier to handle, but it is more poorly behaved then the full crossed
product. In order to construct the reduced norm on $\Gamma_c(G, r^*\A)$, it is necessary to delve into induced representations of groupoid crossed products. Suppose $(\A, G, \alpha)$
is a separable groupoid dynamical system. The space $\Gamma_c(G, s^*\A)$ is a right pre-Hilbert $A$-module with respect to the operations
\[
	z \cdot a(\gamma) = z(\gamma) a(s(\gamma))
\]
and
\[
	\hip{z}{w}_A(u) = \int_G z(\xi)^* w(\xi) \, d\lambda_u(\xi),
\]
and its completion $\Zme$ is a full right Hilbert $A$-module. Furthermore, $\A \rtimes_\alpha G$ acts on $\Zme$ via adjointable operators, where the action 
is characterized by
\[
	f \cdot z(\gamma) = \int_G \alpha_\gamma^{-1} \bigl( f(\eta) \bigr) z(\eta^{-1} \gamma) \, d\lambda^{r(\gamma)}(\eta)
\]
for $f \in \Gamma_c(G, r^*\A)$ and $z \in \Gamma_c(G, s^*\A)$. Thus we can use Rieffel induction to build representations of $\A \rtimes_\alpha G$ from those of $A$. If
$\pi : A \to B(\Hil)$ is a representation of $A$, then we define $\Zme \otimes_A \Hil$ to be the Hilbert space completion of the algebraic tensor product $\Zme \odot \Hil$ with 
respect to the pre-inner product characterized by
\[
	\ip{z \otimes h}{w \otimes k} = \ip{\pi(\hip{w}{z}_A) h}{k}	
\]
for $z, w \in \Zme$ and $h, k \in \Hil$, as in \cite[Equation 2.25]{TFB1}. The induced representation $\Ind \pi$ of $\A \rtimes_\alpha G$ then acts on $\Zme \otimes_A \Hil$ by
\[
	\Ind \pi(f)(z \otimes h) = f \cdot z \otimes h.
\]
We will refer to these sorts of induced representations as \emph{regular representations}. If we take $\pi$ to be faithful, we can define the \emph{reduced norm} by
\[
	\norm{f}_r = \norm{\Ind \pi(f)}.
\]
Notice that this norm is well-defined since induction respects weak containment. The completion of $\Gamma_c(G, r^*\A)$ with respect to $\norm{\cdot}_r$ is called the 
\emph{reduced crossed product} of $\A$ by $G$, denoted by $\A \rtimes_{\alpha, r} G$.

As alluded to above, the full crossed product construction has good functorial properties. On the other hand, the reduced crossed product can be quite poorly behaved at times. 
One well-known instance pertains to short exact sequences. Let $(\A, G, \alpha)$ be a dynamical system, and let $I \subseteq A$ be an ideal. Then \cite[\S3.3]{dana-marius} shows that 
$I$ and $A/I$ are both $C_0(\go)$-algebras, and we let $\I$ and $\A/\I$ denote the corresponding upper semicontinuous $C^*$-bundles. If $I$ is \emph{$G$-invariant}, meaning that
\[
	\alpha_\gamma(\I_{s(\gamma)}) = \I_{r(\gamma)}
\]
for all $\gamma \in G$, then we obtain dynamical systems $(\I, G, \alpha \vert_I)$ and $(\A/\I, G, \alpha^I)$. Furthermore, the resulting sequence of full crossed products
\[
	0 \to \I \rtimes_{\alpha\vert_I} G \to \A \rtimes_{\alpha} G \to \A/\I \rtimes_{\alpha^I} G \to 0
\]
is always exact. In contrast, the reduced crossed product functor is not exact in general. That is, the sequence of reduced crossed products
\begin{equation}
\label{eq:reducedseq}
	0 \to \I \rtimes_{\alpha \vert_I, r} G \to \A \rtimes_{\alpha, r} G \to \A/\I \rtimes_{\alpha^I, r} G \to 0
\end{equation}
need not be exact. Indeed, Higson, Lafforgue, and Skandalis \cite{hls} have produced examples of groupoids for which sequences of the form \eqref{eq:reducedseq} can fail to 
be exact. These so-called HLS groupoids have since been studied in \cite{willett} and \cite{claireweak}. In \cite{willett}, Willett constructed an HLS groupoid $G$ satisfying the weak
containment property (i.e., $C^*(G) = C_r^*(G)$) which nevertheless fails to be amenable. The properties exhibited by the original HLS groupoids in \cite{hls} and by Willett's 
groupoid arise in part due to the failure of the groupoids to be \emph{inner exact}, a term coined by Anantharaman-Delaroche in \cite{claireweak}.

Finally, we will repeatedly appeal to the concept of groupoid equivalence and its relationship to exactness. Recall that two locally compact Hausdorff groupoids $G$ and $H$ are 
said to be \emph{equivalent} if there is a locally compact Hausdorff space $Z$ such that
\begin{itemize}
	\item $G$ and $H$ act freely and properly on the left and right of $Z$, respectively;
	\item the actions of $G$ and $H$ commute; and
	\item the structure maps $Z \to \go$ and $Z \to \ho$ for the actions induce homeomorphisms $Z/H \cong \go$ and $G \backslash Z \cong \ho$.
\end{itemize}
Any such space $Z$ is called a \emph{$(G, H)$-equivalence}. It is a well-known fact that if $G$ and $H$ are equivalent groupoids, then $C^*(G)$ and $C^*(H)$ are Morita 
equivalent, and the same is true for the reduced $C^*$-algebras \cite[Theorem 4.1]{sims-williams2012}. Furthermore, equivalence preserves measurewise amenability 
\cite[Theorem 3.2.16]{ananth-renault} and exactness \cite[Theorem 4.8]{lalondeexact}. We will need to invoke the latter result several times throughout this paper.

\section{Transformation Groupoids}
\label{sec:transformation}
Let $G$ be a locally compact Hausdorff groupoid with Haar system, and suppose $X$ is a left $G$-space. This first part of this section is devoted to showing that any 
$(G \ltimes X)$-dynamical system yields a $G$-dynamical system in a natural way, and the resulting crossed products are isomorphic. As a result, we are then able to show that if 
$G$ is exact and $X$ is any left $G$-space, then the transformation groupoid $G \ltimes X$ is also exact. We use this result in the next section to show that exactness passes to 
closed subgroupoids. 

We first need a fact regarding upper semicontinuous $C^*$-bundles. This result is undoubtedly clear to experts, but we present the details here. The following proposition
is an extension of \cite[Example C.4]{TFB2} to general $C_0(X)$-algebras.

\begin{prop}
\label{prop:c0fiber}
	Let $X$ and $Y$ be locally compact Hausdorff spaces, and suppose $\sigma : Y \to X$ is continuous. Let $\A \to Y$ be an upper semicontinuous $C^*$-bundle over $Y$, and
	let $A = \Gamma_0(Y, \A)$ denote the associated $C_0(Y)$-algebra. Then $A$ is also a $C_0(X)$-algebra, with the $C_0(X)$-action given by
	\[
		(\varphi \cdot a)(y) = \varphi(\sigma(y)) a(y)
	\]
	for $a \in A$ and $\varphi \in C_0(X)$. The resulting upper semicontinuous $C^*$-bundle $\B \to X$ has fibers
	\[
		\B_x = \Gamma_0(\sigma^{-1}(x), \A),
	\]
	and the section algebra $B = \Gamma_0(X, \B)$ is isomorphic to $A$ via the map $\Theta : A \to B$ defined by
	\[
		\Theta(a)(x)(y) = a(x)
	\]
	for $a \in A$, $x \in X$, and $y \in \sigma^{-1}(x)$.
\end{prop}
\begin{proof}
	Since $A$ is a $C_0(Y)$-algebra, there is a continuous map $\tau : \Prim A \to Y$ by \cite[Proposition C.5]{TFB2}. By composing with the map $\sigma : Y \to X$, we obtain
	a continuous map $\sigma \circ \tau : \Prim A \to X$, which makes $A$ into a $C_0(X)$-algebra.
	
	By the discussion on page 355 of \cite{TFB2}, the homomorphism $\Phi_A^X : C_0(X) \to ZM(A)$ associated to the $C_0(X)$-action is defined to be
	\[
		\Phi_A^X(\varphi) = \varphi \circ (\sigma \circ \tau) = (\varphi \circ \sigma) \circ \tau = \Phi_A^Y(\varphi \circ \sigma),
	\]
	where $\Phi_A^Y : C_0(Y) \to ZM(A)$ is the homomorphism implementing the $C_0(Y)$-action on $A$. It is then clear that $\varphi \cdot a = (\varphi \circ \sigma) \cdot a$
	for all $\varphi \in C_0(X)$ and all $a \in A$. 
	
	We will now identify the structure of the fibers of the associated upper semicontinuous bundle $\B \to X$. Given $x \in X$, we have $\B_x = A/I_x$, where
	\[
		I_x = \overline{\spa} \bigl\{ \varphi \cdot a : \varphi \in C_0(X), \, \varphi(x) = 0, \, a \in A \bigr\}.
	\]
	Notice that if $\sigma(y) = x$ and $\varphi(x) = 0$, then
	\[
		(\varphi \cdot a)(y) = \varphi(\sigma(y)) a(y) = \varphi(x) a(y) = 0
	\]
	for all $a \in A$. It follows that $I_x$ consists precisely of the sections in $\Gamma_0(X, \B)$ that vanish on $\sigma^{-1}(x)$, hence $\B_x = \Gamma_0(\sigma^{-1}(x), \A)$. 
	(See \cite[Proposition 5.14]{geoff}, for example.) That $A$ and $B$ are isomorphic follows immediately from \cite[Theorem C.26(c)]{TFB2}, which guarantees the existence
	of a $C_0(X)$-linear isomorphism $\Theta : A \to B$. Given the manner in which we constructed the bundle $\B$, the isomorphism is of course given by restriction.
\end{proof}

We will apply Proposition \ref{prop:c0fiber} in a few different situations, but the first pertains specifically to transformation groupoids. Suppose $X$ is a left $G$-space, and let 
$p : X \to \go$ denote the structure map for the $G$-action. Recall that we assume $p$ is continuous and open. Let $G \ltimes X$ denote the associated transformation groupoid, 
and suppose $(\A, G \ltimes X, \alpha)$ is a separable groupoid dynamical system. Then $\A$ is an upper semicontinuous $C^*$-bundle over $X = (G \ltimes X)^{\scriptscriptstyle{(0)}}$, 
and we let $A = \Gamma_0(X, \A)$ denote the associated $C_0(X)$-algebra of sections. By Proposition \ref{prop:c0fiber}, $A$ is also a $C_0(\go)$-algebra, 
and the resulting upper semicontinuous $C^*$-bundle $\B \to \go$ has fibers $\B_u = \Gamma_0(p^{-1}(u), \A)$ for each $u \in \go$. If we let $B = \Gamma_0(\go, \B)$ denote 
the section algebra, then we also know that $B$ is isomorphic to $A$. In fact, we can use the $(G \ltimes X)$-action on $\A$ to construct a $G$-action on $\B$. 

\begin{prop}
\label{prop:action}
	For each $\gamma \in G$, define $\beta_\gamma : \B_{s(\gamma)} \to \B_{r(\gamma)}$ by
	\[
		\beta_\gamma(b)(x) = \alpha_{(\gamma, x)} \bigl( b(\gamma^{-1} \cdot x) \bigr).
	\]
	Then the family $\{\beta_\gamma \}_{\gamma \in G}$ defines a continuous action of $G$ on $\B$. Consequently, $(\B, G, \beta)$ is a groupoid dynamical system.
\end{prop}

To prove this proposition, we will proceed via a series of lemmas. First we need to make a couple of observations. Let $q : G \times X \to G$ denote the projection onto the first 
coordinate. Since $G \ltimes X \subseteq G \times X$, $q$ restricts to a continuous surjection $G \ltimes X \to G$. As a result, the pullback 
algebras $\Gamma_0(G \ltimes X, s^*\A)$ and $\Gamma_0(G \ltimes X, r^*\A)$ become $C_0(G)$-algebras via Proposition \ref{prop:c0fiber}. Furthermore, for each $\gamma \in G$
the corresponding fiber is
\[
	\Gamma_0(G \ltimes X, r^*\A)(\gamma) = \Gamma_0(q^{-1}(\gamma), r^*\A).
\]
Notice that
\[
	q\vert_{G \ltimes X}^{-1}(\gamma) = \{(\gamma, x) \in G \ltimes X : p(x) = r(\gamma)\} = \{\gamma\} \times p^{-1}(r(\gamma)),
\]
and for each $x \in p^{-1}(r(\gamma))$,
\[
	(r^*\A)_{(\gamma, x)} = \A_{r(\gamma, x)} = \A_x.
\]
Thus
\[
	\Gamma_0(G \ltimes X, r^*\A)(\gamma) = \Gamma_0(p^{-1}(r(\gamma)), \A) = \B_{r(\gamma)}.
\]
It will also be helpful to consider the set
\[
	G * X = \{ (\gamma, x) \in G \times X : s(\gamma) = p(x) \}.
\]
As we mentioned in Section 2, $G * X$ can be made into a groupoid in a natural way. We will not need to consider the groupoid structure on $G*X$, but it is worth noting that 
the unit space can be identified with $X$, and the source map is given by $s(\gamma, x) = x$. Thus we have a continuous surjection $q : G * X \to G$, so 
$\Gamma_0(G * X, s^*\A)$ is a $C_0(G)$-algebra with fibers given by
\[
	\Gamma_0(G*X, s^*\A)(\gamma) = \Gamma_0(q^{-1}(\gamma), s^*\A) = \Gamma_0(p^{-1}(s(\gamma)), \A) = \B_{s(\gamma)}.
\]

Now we can begin the proof. Compare the first lemma to \cite[Lemma 4.37]{geoff} and the implementation of that result in the proof of \cite[Proposition 4.38]{geoff}.

\begin{lem}
\label{lem:lemmaone}
	Let $s^*B = \Gamma_0(G, s^*\B)$ and $r^*B = \Gamma_0(G, r^*\B)$ denote the pullback $C^*$-algebras. There are $C_0(G)$-linear isomorphisms
	\[
		\iota_s : s^*B \to \Gamma_0(G*X, s^*\A), \quad \iota_r : r^*B \to \Gamma_0(G \ltimes X, r^*\A)
	\]
	given by
	\[
		\iota_s(f)(\gamma, x) = f(\gamma)(x), \quad \iota_r(g)(\gamma, x) = g(\gamma)(x)
	\]
	for $f \in s^*B$ and $g \in r^*B$.
\end{lem}
\begin{proof}
	We will work out the proof for $\iota_r$, and things work similarly for $\iota_s$. First notice that if $f \in r^*B$, then $\iota_r(f)$ is clearly a section of $r^*\A$. Moreover, 
	$\iota_r$ is easily seen to be a homomorphism. It is also straightforward to check that $\iota_r$ is isometric: if $f \in r^*B$, then
	\begin{align*}
		\norm{\iota_r(f)}_\infty &= \sup_{(\gamma, x) \in G \ltimes X} \norm{\iota_r(f)(\gamma, x)} \\
			&= \sup_{(\gamma, x) \in G \ltimes X} \norm{f(\gamma)(x)} \\
			&= \sup_{\gamma \in G} \sup_{x \in p^{-1}(r(\gamma))} \norm{f(\gamma)(x)} \\
			&= \sup_{\gamma \in G} \norm{f(\gamma)}_\infty \\
			&= \norm{f}_\infty.
	\end{align*}
	However, it is not clear yet that $\iota_r$ even maps into $\Gamma_0(G \ltimes X, r^*\A)$. That is, we need to check that $\iota_r(f)$ is continuous and vanishes 
	at infinity for all $f \in r^*B$.
	
	Recall that elementary tensors of the form $\varphi \otimes b$ for $\varphi \in C_c(G)$ and $b \in B$, where
	\[
		(\varphi \otimes b)(\gamma) = \varphi(\gamma) b(r(\gamma)),
	\]
	span a dense subalgebra of $r^*B$. (See \cite[\S 3.2]{lalonde2014}, for example.) On such elementary tensors, we have
	\[
		\iota_r(\varphi \otimes b)(\gamma, x) = (\varphi \otimes b)(\gamma)(x) = \varphi(\gamma) b(r(\gamma))(x).
	\]
	Since the function $x \mapsto b(p(x))(x)$ vanishes at infinity by Proposition \ref{prop:c0fiber} (it is precisely $\Theta^{-1}(b) \in A$), the set
	\[
		\{x \in X : \norm{b(p(x))(x)} \geq \varepsilon \}
	\]
	is compact for all $\varepsilon > 0$. Since $\varphi$ is compactly supported on $G$, it is easy to see that $\iota_r(\varphi \otimes b)$ vanishes at infinity on $G \ltimes X$. 
	It remains to see that $\iota_r(\varphi \otimes b)$ is continuous. Suppose $(\gamma_i, x_i) \to (\gamma, x)$ in $G \ltimes X$. Then $\varphi(\gamma_i) \to \varphi(\gamma)$
	and $b(r(\gamma_i))(x_i) \to b(r(\gamma))(x)$ in $\A$, so
	\[
		\bigl( (\gamma_i, x_i), b(r(\gamma_i))(x_i) \bigr) \to \bigl( (\gamma, x), b(r(\gamma))(x) \bigr)
	\]
	in $(G \ltimes X) \times \A$. It follows that 
	\[
		\bigl( (\gamma_i, x_i), \varphi(\gamma_i) b(r(\gamma_i))(x_i) \bigr) \to \bigl( (\gamma, x), \varphi(\gamma) b(r(\gamma))(x) \bigr)
	\]
	in $r^*\A$, so $\iota_r(\varphi \otimes b) \in \Gamma_0(G \ltimes X, r^*\A)$. Since $C_c(G) \odot B$ is dense in $r^*B$ and $\iota_r$ is isometric, it follows that 
	$\iota_r(f)$ is continuous and vanishes at infinity for any $f \in r^*B$.
	
	Now suppose $f \in r^*B$ with $\iota_r(f) = 0$. Then $\iota_r(f)(\gamma, x) = 0$ for all $(\gamma, x) \in G \ltimes X$, so $f(\gamma)(x) = 0$ for all $\gamma \in G$ and $x \in X$.
	Thus $f=0$, and $\iota_r$ is injective. To see that $\iota_r$ is surjective, it suffices to show that the range of $\iota_r$ is closed under the $C_0(G)$-action on 
	$\Gamma_0(G \ltimes X, r^*\A)$ and is fiberwise dense in $r^*\A$. First observe that for each $f \in r^*B$ and $\varphi \in C_0(G)$,
	\[
		\bigl( \varphi \cdot \iota_r(f) \bigr)(\gamma, x) = \varphi(\gamma) f(\gamma)(x) = \bigl( \varphi \cdot f \bigr)(\gamma)(x) = \iota_r(\varphi \cdot f)(\gamma, x),
	\]
	so the range of $\iota_r$ is invariant under the $C_0(G)$-action. Moreover, this computation shows that $\iota_r$ is $C_0(G)$-linear. Now fix $(\gamma, x) \in G \ltimes X$ 
	and let $a \in (r^*\A)_{(\gamma, x)} = \A_x$. Choose $b \in B$ with $b(p(x))(x) = a$, and let $\varphi \in C_c(G)$ with $\varphi(\gamma) = 1$. Then
	\[
		\iota_r(\varphi \otimes b)(\gamma, x) = \varphi(\gamma) b(r(\gamma))(x) = a.
	\]
	Thus the image of $\iota_r$ is fiberwise dense, and it follows from \cite[Proposition C.24]{TFB2} that $\iota_r$ is surjective.
\end{proof}

\begin{lem}
\label{lem:lt}
	The map $\lt : \Gamma_0(G*X, s^*\A) \to \Gamma_0(G \ltimes X, s^*\A)$ defined by
	\[
		\lt(f)(\gamma, x) = f(\gamma, \gamma^{-1} \cdot x)
	\]
	is a $C_0(G)$-linear isomorphism, which restricts to a surjection from $\Gamma_c(G*X, s^*\A)$ onto $\Gamma_c(G \ltimes X, s^*\A)$. Moreover, $\lt$ is 
	continuous with respect to the inductive limit topology.
\end{lem}
\begin{proof}
	We first need to check that $\lt$ actually maps into $\Gamma_0(G \ltimes X, s^*\A)$. To that end, suppose first that 
	$f \in \Gamma_c(G*X, s^*\A)$. It is easy to see that $\lt(f)$ is continuous---if $(\gamma_i, x_i) \to (\gamma, x)$ in $G \ltimes X$, then 
	$(\gamma_i, \gamma_i^{-1} \cdot x_i) \to (\gamma, \gamma^{-1} \cdot x)$ in $G*X$, so
	\[
		\lt(f)(\gamma_i, x_i) = f(\gamma_i, \gamma_i^{-1} \cdot x_i) \to f(\gamma, \gamma^{-1} \cdot x) = \lt(f)(\gamma, x).
	\]
	Now let $K = \supp(f)$. Notice that $\supp(\lt(f))$ is precisely the image of $K$ under the continuous map $(\gamma, x) \mapsto (\gamma, \gamma \cdot x)$ from 
	$G * X \to G \ltimes X$, so it is compact. Thus $\lt(f)$ is compactly supported. 

	Now suppose $f_i \to f$ in $\Gamma_c(G * X, s^*\A)$ with respect to the inductive limit topology. Then $f_i \to f$ uniformly. A straightforward computation like the one from 
	Lemma \ref{lem:lemmaone} shows that $\lt$ is isometric with respect to the supremum norm, so $\lt(f_i) \to \lt(f)$ uniformly as well. Furthermore, there is a compact set $K$ 
	such that $\supp(f_i) \subseteq K$ eventually. The same argument as above shows that the sets $\supp(\lt(f_i))$ are eventually contained in the image of $K$ under the map 
	$(\gamma, x) \mapsto (\gamma, \gamma \cdot x)$, so $\lt(f_i) \to \lt(f)$ in the inductive limit topology.
	
	Since $\lt$ is isometric, it extends to an injective map from $\Gamma_0(G*X, s^*\A)$ to $\Gamma_0(G \ltimes X, s^*\A)$. It is not 
	hard to see that $\lt$ is an isomorphism, since it has an inverse $\lt^{-1} : \Gamma_0(G \ltimes X, s^*\A) \to \Gamma_0(G * X, s^*\A)$ 
	given by
	\[
		\lt^{-1}(f)(\gamma, x) = f(\gamma, \gamma \cdot x).
	\]
	It is also worth noting that arguments like those above show that $\lt^{-1}$ takes compactly supported sections to compactly supported
	sections, and it is continuous in the inductive limit topology. 
	
	All that remains is to see that $\lt$ is $C_0(G)$-linear. Let $f \in \Gamma_0(G*X, s^*\A)$ and $\varphi \in C_0(G)$. Then for all $(\gamma, x) \in G \ltimes X$ we have
	\begin{align*}
		\lt(\varphi \cdot f)(\gamma, x) &= (\varphi \cdot f)(\gamma, \gamma^{-1} \cdot x) \\
			&= \varphi(\gamma) f(\gamma, \gamma^{-1} \cdot x) \\
			&= \varphi(\gamma) \lt(f)(\gamma, x) \\
			&= (\varphi \cdot \lt(f))(\gamma, x). \qedhere
	\end{align*}
\end{proof}

Since $(\A, G \ltimes X, \alpha)$ is a groupoid dynamical system, there is a $C_0(G \ltimes X)$-linear isomorphism $\alpha : s^*A \to r^*A$ given by
$\alpha(f)(\gamma, x) = \alpha_{(\gamma, x)}(f(\gamma, x))$ as a result of \cite[Lemma 4.3]{mw08}.

\begin{lem}
	The isomorphism $\alpha : s^*A \to r^*A$ associated to the $(G \ltimes X)$-action on $\A$ is $C_0(G)$-linear.
\end{lem}
\begin{proof}
	Let $f \in s^*A = \Gamma_0(G \ltimes X, s^*\A)$ and $\varphi \in C_0(G)$. Then
	\begin{align*}
		\alpha( \varphi \cdot f)(\gamma, x) &= \alpha_{(\gamma, x)} \bigl( (\varphi \cdot f)(\gamma, x) \bigr) \\
			&= \alpha_{(\gamma, x)} \bigl( \varphi(\gamma) f(\gamma, x) \bigr) \\
			&= \varphi(\gamma) \alpha_{(\gamma, x)} \bigl( f(\gamma, x) \bigr) \\
			&= \varphi(\gamma) \alpha(f)(\gamma, x) \\
			&= \bigl( \varphi \cdot \alpha(f) \bigr)(\gamma, x)
	\end{align*}
	for all $(\gamma, x) \in G \ltimes X$. Thus $\alpha(\varphi \cdot f) = \varphi \cdot \alpha(f)$, so $\alpha$ is $C_0(G)$-linear.
\end{proof}

\begin{proof}[Proof of Proposition \ref{prop:action}]
	In light of \cite[Lemma 4.3]{mw08}, it suffices to show that the map $\beta : s^*B \to r^*B$ defined by $\beta(f)(\gamma) = \beta_\gamma(f(\gamma))$ is a 
	$C_0(G)$-linear isomorphism. Let $f \in s^*B$. Then $\iota_s(f) \in \Gamma_0(G * X, s^*\A)$ is given by
	\[
		\iota_s(f)(\gamma, x) = f(\gamma)(x),
	\]
	and applying $\lt$ we get
	\[
		(\lt \circ \iota_s)(f)(\gamma, x) = \iota_s(f)(\gamma, \gamma^{-1} \cdot x) = f(\gamma)(\gamma^{-1} \cdot x).
	\]
	Now apply $\alpha$:
	\[
		\alpha \bigl((\lt \circ \iota_s)(f) \bigr)(\gamma, x) = \alpha_{(\gamma, x)} \bigl( (\lt \circ \iota_s)(f)(\gamma, x) \bigr) = \alpha_{(\gamma, x)} \bigl( f(\gamma)(\gamma^{-1} \cdot x) 
		\bigr).
	\]
	Equivalently, we have
	\[
		\alpha \bigl((\lt \circ \iota_s)(f) \bigr)(\gamma, x) = \beta_\gamma(f(\gamma))(x) = \beta(f)(\gamma)(x).
	\]
	Using the fact that the left side is $(\iota_r^{-1} \circ \alpha \circ \lt \circ \iota_s)(f)(\gamma)(x)$, it follows that
	\[
		\beta(f) = (\iota_r^{-1} \circ \alpha \circ \lt \circ \iota_s)(f).
	\]
	Since each of the functions in the composition is a $C_0(G)$-linear isomorphism, it follows that $\beta$ is also a $C_0(G)$-linear isomorphism.
	
	It is also necessary to check that $\beta_\gamma \circ \beta_\eta = \beta_{\gamma \eta}$ whenever $(\gamma, \eta) \in \gtwo$. This is straightforward:
	\begin{align*}
		\beta_\gamma(\beta_\eta(b))(x) &= \alpha_{(\gamma, x)} \bigl( \beta_\eta(b)(\gamma^{-1} \cdot x) \bigr) \\
			&= \alpha_{(\gamma, x)} \bigl( \alpha_{(\eta, \gamma^{-1} \cdot x)} \bigl( b( \eta^{-1} \cdot (\gamma^{-1} \cdot x)) \bigr) \bigr) \\
			&= \alpha_{(\gamma \eta, x)} \bigl( b( (\gamma \eta)^{-1} \cdot x) \bigr) \\
			&= \beta_{\gamma \eta}(b)(x)
	\end{align*}
	for all $b \in \B$. Therefore, $(\B, G, \beta)$ is a groupoid dynamical system.
\end{proof}

Now we have two closely related dynamical systems, $(\A, G \ltimes X, \alpha)$ and $(\B, G, \beta)$. We claim that they yield isomorphic crossed products. This result can 
be thought of as simultaneously generalizing \cite[Proposition 4.38]{geoff} and \cite[Remark 3.61]{jonbrown}.

\begin{thm}
\label{thm:fullcross}
	The map $\Phi : \Gamma_c(G \ltimes X, r^*\A) \to \Gamma_c(G, r^*\B)$ given by
	\[
		\Phi(f)(\gamma)(x) = f(\gamma, x)
	\]
	is a $*$-homomorphism, which extends to an isomorphism $\bar{\Phi} : \A \rtimes_\alpha (G \ltimes X) \to \B \rtimes_\beta G$.
\end{thm}
\begin{proof}
	We will first check that $\Phi$ maps into $\Gamma_c(G, r^*\B)$. Let $f \in \Gamma_c(G \ltimes X, r^*\A)$. Since $\Phi$ agrees with 
	$\iota_r^{-1}$ on $\Gamma_c(G \ltimes X, r^*\A)$, it is immediate that $\Phi(f)$ is continuous. Now let $K$ denote the image of 
	$\supp(f)$ under the projection $G \ltimes X \to G$. If $\gamma \not \in K$, then $\Phi(f)(\gamma) = 0$, so $\supp(\Phi(f))$ is contained in the compact set $K$. Thus
	$\Phi(f)$ is compactly supported.
	
	Since $\Phi$ is clearly linear, we just need to check that it preserves the convolution and involution. If $f, g \in \Gamma_c(G \ltimes X, r^*\A)$, 
	then
	\begin{align*}
		\Phi(f * g)(\gamma)(x) &= (f*g)(\gamma, x) \\
			&= \int_{G \ltimes X} f(\eta, y) \alpha_{(\eta, y)} \bigl( g( (\eta, y)^{-1} (\gamma, x) ) \bigl) \, d(\lambda^{r(\gamma)} \times \delta_x)(\eta, y) \\
			&= \int_G f(\eta, x) \alpha_{(\eta, x)} \bigl( g( (\eta, x)^{-1} (\gamma, x) ) \bigr) \, d\lambda^{r(\gamma)}(\eta) \\
			&= \int_G f(\eta, x) \alpha_{(\eta, x)} \bigl( g(\eta^{-1} \gamma, \eta^{-1} \cdot x) \bigr) \, d\lambda^{r(\gamma)}(\eta) \\
			&= \int_G \Phi(f)(\eta)(x) \alpha_{(\eta, x)} \bigl( \Phi(g)(\eta^{-1} \gamma)(\eta^{-1} \cdot x) \bigr) d\lambda^{r(\gamma)}(\eta) \\
			&= \int_G \Phi(f)(\eta)(x) \beta_\eta \bigl( \Phi(g)(\eta^{-1} \gamma) \bigr)(x) \, d\lambda^{r(\gamma)}(\eta) \\
			&= \Biggl( \int_G \Phi(f)(\eta) \beta_\eta \bigl( \Phi(g)(\eta^{-1} \gamma) \bigr) \, d\lambda^{r(\gamma)}(\eta) \Biggr)(x)
	\end{align*}
	since evaluation at $x$ is a bounded linear map. However,
	\[
		\int_G \Phi(f)(\eta) \beta_\eta \bigl( \Phi(g)(\eta^{-1} \gamma) \bigr) \, d\lambda^{r(\gamma)}(\eta) = \Phi(f) * \Phi(g) (\gamma),
	\]
	and it follows that $\Phi$ is multiplicative. As for the involution, we have
	\begin{align*}
		\Phi(f^*)(\gamma)(x) &= f^*(\gamma, x) \\
			&= \alpha_{(\gamma, x)} \bigl( f((\gamma, x)^{-1})^* \bigr) \\
			&= \alpha_{(\gamma, x)} \bigl( f(\gamma^{-1}, \gamma^{-1} \cdot x)^* \bigr) \\
			&= \alpha_{(\gamma,x)} \bigl( \Phi(f)(\gamma^{-1})(\gamma^{-1} \cdot x) \bigr)^* \\
			&= \beta_\gamma \bigl( \Phi(f)(\gamma^{-1})^* \bigr)(x) \\
			&= \Phi(f)^*(\gamma)(x).
	\end{align*}
	Thus $\Phi$ is a $*$-homomorphism.
	
	To see that $\Phi$ extends to a homomorphism between the full crossed products, we show it is continuous with respect to the inductive limit topology. Suppose $f_i \to f$ in 
	$\Gamma_c(G \ltimes X, r^*\A)$ with respect to the inductive limit topology. Then $f_i \to f$ uniformly, and there is a fixed compact set $K_0 \subseteq G \ltimes X$ such that
	$\supp(f_i) \subseteq K_0$ eventually. Let $K \subseteq G$ denote the image of $K_0$ under projection onto the first coordinate. Then $K$ is compact. Moreover, $\supp(f_i)
	\subseteq K_0$ implies $\Phi(f_i)(\gamma) = f(\gamma, \cdot) = 0$ when $\gamma \not \in K$. Therefore, $\supp(\Phi(f_i))$ is eventually contained in $K$. Also,
	\begin{align*}
		\norm{\Phi(f_i) - \Phi(f)}_\infty &= \sup_{\gamma \in G} \norm{\Phi(f_i)(\gamma) - \Phi(f)(\gamma)}_\infty \\
			&= \sup_{\gamma \in G} \sup_{x \in p^{-1}(r(\gamma))} \norm{\Phi(f_i)(\gamma)(x) - \Phi(f)(\gamma)(x)} \\
			&= \sup_{\gamma \in G} \sup_{x \in p^{-1}(r(\gamma))} \norm{f_i(\gamma, x) - f(\gamma, x)} \\
			&= \sup_{(\gamma, x) \in G \ltimes X} \norm{f_i(\gamma, x) - f(\gamma, x)} \\
			&= \norm{f_i - f}_\infty,
	\end{align*}
	so $\Phi(f_i) \to \Phi(f)$ uniformly. Thus $\Phi(f_i) \to \Phi(f)$ in the inductive limit topology, so $\Phi$ is continuous. It follows that $\Phi$ extends to a homomorphism
	$\bar{\Phi} : \A \rtimes_\alpha (G \ltimes X) \to \B \rtimes_\beta G$. It remains to see that $\bar{\Phi}$ is an isomorphism.
	
	We begin by showing that $\Phi$ is injective on $\Gamma_c(G \ltimes X, r^*\A)$. This is fairly routine: if $\Phi(f) = 0$ for some $f \in \Gamma_c(G \ltimes X, r^*\A)$, then
	$\Phi(f)(\gamma) = 0$ for all $\gamma \in G$, which in turn means that $\Phi(f)(\gamma)(x) = 0$ for all $\gamma \in G$ and all $x \in p^{-1}(r(\gamma))$. But then
	\[
		f(\gamma, x) = \Phi(f)(\gamma)(x) = 0
	\]
	for all $(\gamma, x) \in G \ltimes X$. Thus $f=0$, and $\Phi$ is injective on $\Gamma_c(G \ltimes X, r^*\A)$. Thus $\bar{\Phi}$ restricts to an isomorphism of 
	$\Gamma_c(G \ltimes X, r^*\A)$ onto $\image{\Phi} \subseteq \Gamma_c(G, r^*\B)$.
	
	Next we claim that $\image \Phi$ is dense in $\Gamma_c(G, r^*\B)$ with respect to the inductive limit topology, which will show that $\bar{\Phi}$ is surjective. We will use 
	\cite[Proposition C.24]{TFB2}. Notice first that $\image{\Phi}$ is a $C_0(G)$-module: if $g = \Phi(f)$ for some $f \in \Gamma_c(G \ltimes X, r^*\A)$ and $\varphi \in C_0(G)$,
	then
	\begin{align*}
		\Phi(\varphi \cdot f)(\gamma)(x) &= (\varphi \cdot f)(\gamma, x) \\
			&= \varphi(\gamma) f(\gamma, x) \\
			&= \varphi(\gamma) \Phi(f)(\gamma)(x) \\
			&= (\varphi \cdot \Phi(f))(\gamma)(x) \\
			&= (\varphi \cdot g)(\gamma)(x),
	\end{align*}
	so $\varphi \cdot g \in \image \Phi$. Now fix $\gamma \in G$. We need to show that
	\[
		\{ g(\gamma) : g \in \image \Phi \}
	\]
	is dense in $r^*\B_\gamma$. Let $b \in \Gamma_c(p^{-1}(r(\gamma)), \A) \subseteq r^*\B_\gamma$, and use the vector-valued Tietze extension theorem 
	\cite[Proposition A.5]{muhly-williams} to find $f \in \Gamma_c(G \ltimes X, \A)$ such that $f(\gamma, x) = b(x)$ for all $x \in p^{-1}(r(\gamma))$. Then
	\[
		\Phi(f)(\gamma)(x) = f(\gamma, x) = b(x)
	\]
	for all $x \in p^{-1}(r(\gamma))$, so $\Phi(f)(\gamma) = b$. It follows that $\{g(\gamma) : g \in \image \Phi\}$ is dense in $(r^*\B)_\gamma$, so \cite[Proposition C.24]{TFB2} implies
	that $\image \Phi$ is dense in $\Gamma_0(G, r^*\B)$. Now let $g \in \Gamma_c(G, r^*\B)$. Then there is a net $f_i \in \Gamma_c(G \ltimes X, r^*\A)$ such that $\Phi(f_i) \to g$ 
	uniformly. Let $K = \supp(g)$, and choose $\varphi \in C_c(G)^+$ such that $\varphi \vert_K \equiv 1$ and $\varphi(x) < 1$ for all $x \not\in K$. Put $g_i = \varphi \cdot \Phi(f_i) 
	= \Phi(\varphi \cdot f_i)$. Notice that $\varphi \cdot g = g$, so $g_i \to g$ uniformly. Moreover, each $g_i$ is compactly supported, and $\supp(g_i) \subseteq \supp(\varphi)$ for all 
	$i$. Thus $g_i \to g$ in the inductive limit topology. It follows that $\bar{\Phi}$ is surjective.
	
	Recall that $\bar{\Phi}$ restricts to an isomorphism of $\Gamma_c(G \ltimes X, r^*\A)$ onto the dense subalgebra $\image \Phi$, so we have an inverse 
	$\Psi : \image \Phi \to \Gamma_c(G \ltimes X, r^*\A)$ given by
	\[
		\Psi(g)(\gamma, x) = g(\gamma)(x).
	\]
	We claim that $\Psi$ extends to a homomorphism at the level of crossed products---to show it, we will prove that $\Psi$ is $I$-norm decreasing. Let $g \in \image \Phi$, 
	and observe that
	\begin{align*}
		\int_{G \ltimes X} \norm{\Psi(g)(\gamma, y)} \, d(\lambda^{r(x)} \times \delta_x)(\gamma, y) &= \int_G \norm{\Psi(g)(\gamma, x)} \, d\lambda^{r(x)}(\gamma) \\
			&= \int_G \norm{g(\gamma)(x)} \, d\lambda^{r(x)}(\gamma) \\
			&\leq \int_G \norm{g(\gamma)}_\infty \, d\lambda^{r(x)}(\gamma) \\
			&\leq \norm{g}_I.
	\end{align*}
	Similarly,
	\begin{align*}
		\int_{G \ltimes X} \norm{\Psi(g)(\gamma^{-1}, \gamma^{-1} \cdot y)} \, d(\lambda^{r(x)} \times \delta_x)(\gamma, y) 
			\leq \int_G \norm{g(\gamma^{-1})}_\infty \, d\lambda^{r(x)}(\gamma) \leq \norm{g}_I,
	\end{align*}
	and taking suprema leads to the conclusion that $\norm{\Psi(g)}_I \leq \norm{g}_I$. Since $\Psi$ is $I$-norm decreasing, it extends to a homomorphism 
	$\bar{\Psi} : \B \rtimes_\beta G \to \A \rtimes_\alpha (G \ltimes X)$. Moreover, $\bar{\Psi}$ and $\bar{\Phi}$ are inverses on dense subalgebras, so they must be 
	inverses everywhere. That is, $\bar{\Phi}$ is an isomorphism.
\end{proof}

In order to prove results about exactness, we require an analogue of Theorem \ref{thm:fullcross} for the \emph{reduced} crossed product. Therefore, we need to show 
the homomorphism $\Phi : \Gamma_c(G \ltimes X, r^*\A) \to \Gamma_c(G, r^*\B)$ is isometric with respect to the reduced norms. Let $\pi : B \to B(\Hil)$ be a faithful 
representation, and let $\Y$ denote the right Hilbert $B$-module obtained by completing $\Gamma_c(G, s^*\B)$ with respect to the $B$-valued inner product defined
by
\[
	\hip{w}{z}_B(u) = \int_G w(\gamma)^* z(\gamma) \, d\lambda_u(\gamma)
\]
for all $u \in \go$. Then the induced representation $\Ind \pi$ acts on the Hilbert space $\Y \otimes_B \Hil$, which is equipped with the inner product
\[
	\ip{z \otimes h}{w \otimes k} = \ip{\pi \bigl( \hip{w}{z}_B \bigr) h}{k},
\]
Moreover,
\[
	\norm{f}_r = \norm{\Ind \pi(f)}
\]
for all $f \in \Gamma_c(G, r^*\B)$. Likewise, $\pi \circ \Theta : A \to B(\Hil)$ is a faithful representation, where $\Theta : A \to B$ is the isomorphism afforded by Proposition 
\ref{prop:c0fiber}. If we let $\Zme$ denote the right Hilbert $A$-module completion of $\Gamma_c(G \ltimes X, s^*\A)$ with respect to the inner product defined by
\[
	\hip{w}{z}_A(x) = \int_G w(\gamma, \gamma \cdot x)^* z(\gamma, \gamma \cdot x) \, d\lambda_{p(x)}(\gamma)
\]
for all $x \in X$, then the resulting induced representation $\Ind(\pi \circ \Theta)$ acts on $\Zme \otimes_A \Hil$, which carries the inner product
\[
	\ip{z \otimes h}{w \otimes k} = \ip{\pi \circ \Theta \bigl( \hip{w}{z}_A \bigr) h}{k}.
\]
Again,
\[
	\norm{f}_r = \norm{\Ind(\pi \circ \Theta)(f)}
\]
for all $f \in \Gamma_c(G \ltimes X, r^*\A)$. Therefore, to show that $\Phi$ is isometric, it is enough to show that $(\Ind \pi) \circ \Phi$ and $\Ind(\pi \circ \Theta)$ are unitarily
equivalent. Notice that the map $\lt^{-1}$ clearly takes compactly supported sections to compactly supported sections, and an argument like the one from the proof of the previous 
theorem shows the same is true for $\iota_s^{-1}$. Thus we have a natural map from $\Gamma_c(G \ltimes X, s^*\A)$ to $\Gamma_c(G, s^*\B)$, namely $(\lt \circ \iota_s)^{-1}$. 
We then define a map $U : \Gamma_c(G \ltimes X, s^*\A) \odot \Hil \to \Gamma_c(G, s^*\B) \odot \Hil$ on elementary tensors by
\[
	U(z \otimes h) = (\lt \circ \iota_s)^{-1}(z) \otimes h,
\]
which we claim extends to the completion $\Zme \otimes_A \Hil$.

\begin{prop}
	The map $U$ extends to a unitary $U : \Zme \otimes_A \Hil \to \Y \otimes_B \Hil$ that intertwines the representations
	$(\Ind \pi) \circ \Phi$ and $\Ind(\pi \circ \Theta)$.
\end{prop}
\begin{proof}
	To see that $U$ extends to a unitary, we just need to check that $U$ is isometric and has dense range. Let $z \otimes h, w \otimes k \in \Gamma_c(G \ltimes X, s^*\A) \odot \Hil$
	and observe that 
	\begin{align*}
		\ip{U(z \otimes h)}{U(w \otimes k)} &= \ip{(\lt \circ \iota_s)^{-1}(z) \otimes h}{(\lt \circ \iota_s)^{-1}(w) \otimes k} \\
			&= \ip{\pi \bigl( \hip{(\lt \circ \iota_s)^{-1}(w)}{(\lt \circ \iota_s)^{-1}(z)}_B \bigr) h}{k},
	\end{align*}
	where
	\begin{align*}
		& \hip{(\lt \circ \iota_s)^{-1}(w)}{(\lt \circ \iota_s)^{-1}(z)}_B(u)(x) \\
			&\qquad = \biggl( \int_G (\lt \circ \iota_s)^{-1}(w)(\gamma)^* (\lt \circ \iota_s)^{-1}(z)(\gamma) \, d\lambda_u(\gamma) \biggr)(x) \\
			&\qquad = \int_G (\lt \circ \iota_s)^{-1}(w)(\gamma)(x)^* (\lt \circ \iota_s)^{-1}(z)(\gamma)(x) \, d\lambda_u(\gamma) \\
			&\qquad = \int_G \lt^{-1}(w)(\gamma, x)^* \lt^{-1}(z)(\gamma, x) \, d\lambda_u(\gamma) \\
			&\qquad = \int_G w(\gamma, \gamma \cdot x)^* z(\gamma, \gamma \cdot x) \, d\lambda_{p(x)}(\gamma) \\
			&\qquad = \int_{G \ltimes X} w(\gamma, y)^* z(\gamma, y) \, d\mu_x(\gamma, y) \\
			&\qquad = \hip{w}{z}_A(x)
	\end{align*}
	for all $u \in \go$ and all $x \in p^{-1}(u)$. Notice that 
	\[
		\Theta( \hip{w}{z}_A )(u, x) = \hip{w}{z}_A(x),
	\]
	so it follows that
	\[
		\Theta( \hip{w}{z}_A ) = \hip{(\lt \circ \iota_s)^{-1}(w)}{(\lt \circ \iota_s)^{-1}(z)}_B.
	\]
	Thus
	\begin{align*}
		\ip{U(z \otimes h)}{U(w \otimes k)} &= \ip{\pi \circ \Theta \bigl( \hip{w}{z}_A \bigr) h}{k} \\
			&= \ip{z \otimes h}{w \otimes k}
	\end{align*}
	in $\Gamma_c(G \ltimes X, s^*\A) \odot \Hil$, so $U$ is isometric. To see that it has dense range, we first claim that $(\lt \circ \iota_s)^{-1}$ 
	is continuous with respect to the inductive limit topologies on $\Gamma_c(G \ltimes X, s^*\A)$ and $\Gamma_c(G, s^*\B)$. This is fairly straightforward---an 
	argument like the one from the proof of Theorem \ref{thm:fullcross} shows that $\iota_s^{-1}$ is continuous with respect to the inductive limit topology, and we 
	already established that $\lt^{-1}$ is continuous in the inductive limit topology in the proof of Lemma \ref{lem:lt}. Thus the range of $(\lt \circ \iota_s)^{-1}$ 
	is dense in $\Gamma_c(G, s^*\B)$ with respect to the inductive limit topology. An argument identical to the one from \cite[Lemma 5.5]{lalonde2014} then shows that
	the range of $(\lt \circ \iota_s)^{-1}$ is norm-dense in the Hilbert $B$-module $\Y$. It follows that the range of $U$ is dense in 
	$\Y \otimes_B \Hil$: if $z_i \to z$ in $\Y$, then for any $h \in \Hil$ we have
	\begin{align*}
		\norm{(z_i - z) \otimes h}^2 &= \ip{(z_i - z) \otimes h}{(z_i - z) \otimes h} \\
			&= \ip{\pi \bigl( \hip{z_i - z}{z_i - z}_B \bigr)h}{h} \\
			&\leq \norm{\hip{z_i - z}{z_i - z}_B} \norm{h}^2 \\
			&= \norm{z_i - z}^2 \norm{h}^2 \\
			&\to 0,
	\end{align*}
	so $z_i \otimes h \to z \otimes h$. Therefore, $U$ extends to a unitary.
	
	We are left with the verification that $U$ intertwines $(\Ind \pi) \circ \Phi$ and $\Ind(\pi \circ \Theta)$. Well, observe that if $f \in 
	\Gamma_c(G \ltimes X, r^*\A)$ and $z \otimes h \in \Gamma_c(G \ltimes X, s^*\A) \odot \Hil$, then
	\begin{align*}
		\Ind \pi (\Phi(f)) \cdot U(z \otimes h) &= \Ind \pi(\Phi(f)) \bigl( (\lt \circ \iota_s)^{-1}(z) \otimes h \bigr) \\
			&= \Phi(f) \cdot (\lt \circ \iota_s)^{-1}(z) \otimes h,
	\end{align*}
	where
	\begin{align*}
		\Phi(f) \cdot (\lt \circ \iota_s)^{-1}&(z)(\gamma)(x) = \int_G \beta_\gamma^{-1} \bigl( \Phi(f)(\eta) \bigr)(x) (\lt \circ \iota_s)^{-1}(z)
				(\eta^{-1} \gamma)(x) \, d\lambda^{r(\gamma)}(\eta) \\
			&= \int_G \alpha_{(\gamma^{-1}, x)} \bigl( \Phi(f)(\eta)(\gamma \cdot x) \bigr) \lt^{-1}(z)(\eta^{-1}\gamma, x) \, 
				d\lambda^{r(\gamma)}(\eta) \\
			&= \int_G \alpha_{(\gamma, \gamma \cdot x)}^{-1} \bigl( f(\eta, \gamma \cdot x) \bigr) z(\eta^{-1} \gamma, \eta^{-1} \gamma \cdot x) 
				\, d\lambda^{r(\gamma)}(\eta) \\
			&= \int_G \alpha_{(\gamma, \gamma \cdot x)}^{-1} \bigl( f(\eta, \gamma \cdot x) \bigr) z \bigl( (\eta^{-1}, \eta^{-1}\gamma \cdot x)
				(\gamma, \gamma \cdot x) \bigr) \, d\lambda^{r(\gamma)}(\eta) \\
			&= \int_G \alpha_{(\gamma, \gamma \cdot x)}^{-1} \bigl( f(\eta, \gamma \cdot x) \bigr) z \bigl( (\eta, \gamma \cdot x)^{-1} (\gamma, 
				\gamma \cdot x) \bigr) \, d\lambda^{r(\gamma)}(\eta) \\
			&= f \cdot z(\gamma, \gamma \cdot x) \\
			&= \lt^{-1}(f \cdot z)(\gamma, x) \\
			&= \iota_s^{-1} \bigl( \lt^{-1}(f \cdot z) \bigr)(\gamma)(x).
	\end{align*}
	That is,
	\[
		\Phi(f) \cdot (\lt \circ \iota_s)^{-1}(z) = (\lt \circ \iota_s)^{-1}(f \cdot z),
	\]
	so
	\begin{align*}
		\Ind \pi (\Phi(f)) \cdot U(z \otimes h) &= \Phi(f) \cdot (\lt \circ \iota_s)^{-1}(z) \otimes h \\
			&= (\lt \circ \iota_s)^{-1}(f \cdot z) \otimes h \\
			&= U(f \cdot z \otimes h) \\
			&= U \cdot \Ind(\pi \circ \Theta)(f)(z \otimes h).
	\end{align*}
	Thus $U$ indeed intertwines the given representations.
\end{proof}

\begin{thm}
\label{thm:redcross}
	The map $\Phi : \Gamma_c(G \ltimes X, r^*\A) \to \Gamma_c(G, r^*\B)$ from Theorem \ref{thm:fullcross} extends to an isomorphism
	$\Phi : \A \rtimes_{\alpha, r} (G \ltimes X) \to \B \rtimes_{\beta, r} G$.
\end{thm}
\begin{proof}
	We just need to check that $\Phi$ is isometric with respect to the reduced norms. This follows easily from the last proposition: if 
	$\pi$ is a faithful representation of $B$ and $f \in \Gamma_c(G \ltimes X, r^*\A)$, then
	\[
		\norm{\Phi(f)}_r = \norm{\Ind \pi(\Phi(f))} = \norm{\Ind(\pi \circ \Theta)(f)} = \norm{f}_r. \qedhere
	\]
\end{proof}

Before we can prove the main theorem of this section, we need one final result about invariant ideals.

\begin{prop}
	Let $\Theta : A \to B$ denote the isomorphism afforded by Proposition \ref{prop:c0fiber}. Let $I \subseteq A$ be an ideal, and let $J = \Theta(I)$. Then $I$ is 
	$(G \ltimes X)$-invariant if and only if $J$ is $G$-invariant.
\end{prop}
\begin{proof}
	Let $\I$ and $\mathcal{J}$ denote the upper semicontinuous $C^*$-bundles associated to $I$ and $J$. Suppose first that $I$ is $(G \ltimes X)$-invariant, i.e., for all 
	$(\gamma, x) \in G \ltimes X$,
	\[
		\alpha_{(\gamma, x)} \bigl( \I_{\gamma^{-1} \cdot x} \bigr) = \I_x.
	\]
	Observe that for each $\gamma \in G$, $\mathcal{J}_{s(\gamma)} = C_0(p^{-1}(s(\gamma)), \I)$. Thus if $b \in \mathcal{J}_{s(\gamma)}$ we have
	\[
		\beta_\gamma(b)(x) = \alpha_{(\gamma, x)} \bigl( b(\gamma^{-1} \cdot x) \bigr) \in \I_x,
	\]
	so $\beta_\gamma(b) \in C_0(p^{-1}(r(\gamma)), \I) = \mathcal{J}_{r(\gamma)}$. It is not hard to see that $\beta_\gamma(\mathcal{J}_{s(\gamma)}) = \mathcal{J}_{r(\gamma)}$,
	since the same argument shows that $\beta_\gamma^{-1}(\mathcal{J}_{r(\gamma)}) \subseteq \mathcal{J}_{s(\gamma)}$. Thus $J$ is a $G$-invariant ideal.
	
	Conversely, suppose that $J$ is $G$-invariant. Let $a \in \I_{\gamma^{-1} \cdot x}$, and choose $b \in \mathcal{J}_{s(\gamma)} = C_0(p^{-1}(s(\gamma)), \I)$ such that 
	$b(\gamma^{-1} \cdot x) = a$. Then we have
	\[
		\alpha_{(\gamma, x)}(a) = \alpha_{(\gamma, x)} \bigl( b(\gamma^{-1} \cdot x) \bigr) = \beta_\gamma(b)(x).
	\]
	Since $J$ is $G$-invariant, $\beta_\gamma(b) \in \mathcal{J}_{r(\gamma)} = C_0(p^{-1}(r(\gamma)), \I)$. Thus $\beta_\gamma(b)(x) \in \I_x$, so $\alpha_{(\gamma, x)}
	(\I_{\gamma^{-1} \cdot x}) \subseteq \I_x$. Again, it is not hard to show that $\alpha_{(\gamma, x)}$ maps $\I_{\gamma^{-1} \cdot x}$ onto $\I_x$, so $I$ is 
	$(G \ltimes X)$-invariant.
\end{proof}

\begin{prop}
\label{prop:exactcheck}
	Let $(\A, G \ltimes X, \alpha)$ and $(\B, G, \beta)$ be as above. Let $I \subseteq A$ be a $(G \ltimes X)$-invariant ideal, and let $J = \Theta(I)$ be the corresponding
	$G$-invariant ideal in $B$. Then the sequence
	\[
		0 \to \mathcal{J} \rtimes_{\beta, r} G \to \B \rtimes_{\beta, r} G \to \B/\mathcal{J} \rtimes_{\beta, r} G \to 0
	\]
	is exact if and only if
	\[
		0 \to \I \rtimes_{\alpha, r} (G \ltimes X) \to \A \rtimes_{\alpha, r} (G \ltimes X) \to \A/\I \rtimes_{\alpha, r} (G \ltimes X) \to 0
	\]
	is exact.
\end{prop}
\begin{proof}
	It is straightforward to check that the diagram
	\[
		\xymatrix@C=2em{ 0 \ar[r] & \mathcal{J} \rtimes_{\beta, r} G \ar[r] \ar[d] & \B \rtimes_{\beta, r} G \ar[r] \ar[d] & \B/\mathcal{J} \rtimes_{\beta, r} G \ar[r] \ar[d] & 0 \\
		0 \ar[r] & \I \rtimes_{\alpha, r} (G \ltimes X) \ar[r] & \A \rtimes_{\alpha, r} (G \ltimes X) \ar[r] & \A/\I \rtimes_{\alpha, r} (G \ltimes X) \ar[r] & 0
		}
	\]
	commutes, where the vertical arrows are the isomorphisms coming from Theorem \ref{thm:redcross}. Thus one sequence is exact if and only if the other is.
\end{proof}

\begin{thm}
\label{thm:exacttrans}
	Let $G$ be a locally compact Hausdorff groupoid, and let $X$ be a left $G$-space. If $G$ is exact, then the transformation groupoid 
	$G \ltimes X$ is also exact.
\end{thm}
\begin{proof}
	Let $(\A, G \ltimes X, \alpha)$ be a separable groupoid dynamical system, and let $I \subseteq A = \Gamma_0(X, \A)$ be a $(G \ltimes X)$-invariant ideal. Then there is a 
	dynamical system $(\B, G, \beta)$ with $\A \rtimes_{\alpha, r} (G \ltimes X) \cong \B \rtimes_{\beta, r} G$ and an isomorphism $\Theta : A \to B = \Gamma_0(\go, \B)$. The 
	ideal $J = \Theta(I)$ of $B$ is $G$-invariant, and Proposition \ref{prop:exactcheck} gives us an isomorphism of short sequences. If $G$ is exact, then the top sequence is 
	exact, so the bottom one must be exact as well. It follows that $G \ltimes X$ is exact.
\end{proof}

\section{Subgroupoids}
\label{sec:subgroupoids}
Throughout this section, let $G$ denote a locally compact Hausdorff groupoid. We say that a subgroupoid $H \subseteq G$ is \emph{wide} if $\ho = \go$. As an interesting corollary 
of Theorem \ref{thm:exacttrans}, we show that if $G$ is exact, then any wide subgroupoid of $G$ is also exact. We then move on to other types of subgroupoids, which require
different techniques.

Let $H \subseteq G$ be a closed, wide subgroupoid with open range and source maps. Then $H$ acts freely and properly on 
$G$ by right translation, so the orbit space $G/H$ is locally compact and Hausdorff. As we will see in the next result, $G$ acts naturally on $G/H$ by left translation, so we can form 
the transformation groupoid $G \ltimes G/H$. Moreover, there is a natural free and proper action of $G \ltimes G/H$ on the left of $G$, which commutes with the natural $H$-action, 
and we have $(G \ltimes G/H) \backslash G \cong \go = \ho$ and $G/H \cong (G \ltimes G/H)^{\scriptscriptstyle{(0)}}$. In other words, $G$ implements an \emph{equivalence} between 
the groupoids $H$ and $G \ltimes G/H$. This is particularly noteworthy, since groupoid equivalence is known to preserve exactness \cite[Theorem 4.8]{lalondeexact}.

\begin{thm}
\label{thm:wideequiv}
	Let $G$ be a locally compact Hausdorff groupoid with a Haar system, and let $H \subseteq G$ be a closed, wide subgroupoid with open range and source maps. Then the 
	space $G$ is a $(G \ltimes G/H, H)$-equivalence.
\end{thm}
\begin{proof}
	We have already mentioned that $H$ acts freely and properly on the right of $G$. The structure map $r_{G} : G \to G/H$ for the left 
	$G \ltimes G/H$-action is given by $r_{G}(x) = [x]$. Notice that this map is open by \cite[Proposition 5.27]{muhly} and clearly descends to a homeomorphism 
	$G/H \to G/H$ given by the identity map. Observe then that $(\gamma, [\eta]) \in G \ltimes G/H$ and $x \in G$ are composable whenever 
	$s(\gamma, [\eta]) = [\gamma^{-1} \eta]$ equals $r_G(x) = [x]$, which implies that $r(x) = s(\gamma)$. Thus left translation in $G$ induces an action of $G \ltimes G/H$ on 
	$G$:
	\[
		(\gamma, [\eta]) \cdot x = \gamma x.
	\]
	To see that the action is free, suppose $(\gamma, [\eta]) \cdot x = x$ for some $x \in G$ and $(\gamma, [\eta]) \in G \ltimes G/H$. Then we have $\gamma x = x$,
	which implies $\gamma = r(x)$. Thus $(\gamma, [\eta]) = (r(x), [\eta])$ is a unit in $G \ltimes G/H$, so the action is free. 
	
	To prove the action is proper, we appeal to \cite[Proposition 1.84]{geoff}. Suppose $\{x_i\}$ and $\{(\gamma_i, [\eta_i])\}$ are nets in $G$ and $G \ltimes G/H$, respectively, 
	such that $x_i \to x$ and $(\gamma_i, [\eta_i]) \cdot x_i \to y$ for some $x, y \in G$. Note that the latter condition really says that $\gamma_i x_i \to y$, so we can 
	conclude that $\{\gamma_i\}$ has a convergent subnet by \cite[Proposition 1.84]{geoff} since $G$ acts properly on itself by left translation. Pass to this subnet, relabel,
	and observe that since $[x_i] = [\gamma_i^{-1} \eta_i]$ for all $i$ and $x_i \to x$, we have $[\gamma_i^{-1} \eta_i] \to [x]$. But then since $\gamma_i \to \gamma$ for
	some $\gamma \in G$, we have
	\[
		\gamma_i \cdot [\gamma_i^{-1} \eta_i] \to \gamma \cdot [x] = [\gamma x],
	\]
	or $[\eta_i] \to [\gamma x]$. Thus the subnet $\{(\gamma_i, [\eta_i])\}$ converges to $(\gamma, \gamma x)$, so $G$ is a proper left $(G \ltimes G/H)$-space.
	
	It is straightforward to see that the actions of $G \ltimes G/H$ and $H$ on $G$ commute, so we just need to check that the structure map $s_G : G \to \ho = \go$
	descends to a homeomorphism of $(G \ltimes G/H) \backslash G$ with $\ho$. Indeed, since $s_G$ is open it suffices to see that it induces a bijection. Certainly if 
	$x, y \in G$ lie in the same $(G \ltimes G/H)$-orbit, then $s_G(x) = s_G(y)$, so $s_G$ does induce a well-defined, surjective map 
	$\tilde{s}_G : (G \ltimes G/H) \backslash G \to \ho$. Now suppose $s_G(x) = s_G(y)$ for some $x, y \in G$. Then we define $\gamma = yx^{-1}$, and observe that
	\[
		(\gamma, [\gamma x]) \cdot x = \gamma x = y,
	\]
	so $x$ and $y$ belong to the same $(G \ltimes G/H)$-orbit. Thus $\tilde{s}_G$ is injective. It follows that $\tilde{s}_G$ is a homeomorphism. Therefore, $G$ is a
	$(G \ltimes G/H, H)$-equivalence.
\end{proof}

In order to discuss exactness for a subgroupoid $H$, we need to require that $H$ is equipped with a Haar system. As a consequence of Theorem \ref{thm:wideequiv},
together with \cite[Theorem 2.1]{danaHaar}, it suffices to require that $H$ has open range and source maps.

\begin{thm}
\label{thm:haar}
	Let $G$ be a locally compact Hausdorff groupoid with a Haar system, and let $H \subseteq G$ be a closed, wide subgroupoid with open range and source maps.
	Then $H$ possesses a Haar system.
\end{thm}
\begin{proof}
	If $G$ has a Haar system, then the transformation groupoid $G \ltimes G/H$ also has a Haar system. Since the property of having a Haar system is preserved under equivalence
	of groupoids by \cite[Theorem 2.1]{danaHaar}, it follows from Theorem \ref{thm:wideequiv} that $H$ has a Haar system.
\end{proof}

\begin{thm}
\label{thm:wideexact}
	Let $G$ be a locally compact Hausdorff groupoid with a Haar system, and let $H$ be a closed, wide subgroupoid with open range and source maps. If $G$ is exact, then $H$ 
	is exact.
\end{thm}
\begin{proof}
	If $G$ is exact, then Theorem \ref{thm:exacttrans} guarantees the transformation groupoid $G \ltimes G/H$ is exact. But $G \ltimes G/H$ is equivalent to $H$ by
	Theorem \ref{thm:wideequiv}, which implies that $H$ is also exact by \cite[Theorem 4.8]{lalondeexact}.
\end{proof}

\begin{cor}
\label{cor:isoexact}
	Let $G$ be a locally compact Hausdorff groupoid with a Haar system, and assume $G$ has continuously varying stabilizers. Then the isotropy groupoid $\iso(G)$
	is exact.
\end{cor}
\begin{proof}
	Recall that $G$ has continuously varying stabilizers if and only if $\iso(G)$ has open range and source maps, or equivalently if $\iso(G)$ has a Haar system. Thus
	$\iso(G)$ is a wide subgroupoid of $G$ with open range and source maps, so it is exact by Theorem \ref{thm:wideexact}.
\end{proof}

Under certain conditions, we can extend Theorem \ref{thm:wideexact} slightly. Suppose $H$ is only \emph{orbit-wide}, meaning $\ho$ meets every orbit in $\go$. Put 
$X = s^{-1}(\ho)$. As observed in \cite[Example 1.63]{geoff}, $X$ is saturated with respect to the source map, so the restriction $s : X \to \ho$ is open. It is then easy to see that the 
reduction $G \vert_{\ho}$ acts freely and properly on the right of $X$. If the restriction of the range map to $X$ also happens to be open, then $X$ is also a free and proper left 
$G$-space. Indeed, $X$ becomes a $(G,G \vert_\ho)$-equivalence, as observed in \cite[Example 5.33(7)]{muhly}. However, Muhly also notes in \cite[Example 5.29(2)]{muhly} that 
the restriction $r : X \to \go$ need not be open in general. In cases where it is open, we have the following result.

\begin{thm}
	Let $G$ be a locally compact Hausdorff groupoid with a Haar system, and suppose $H \subseteq G$ is a closed, orbit-wide subgroupoid with open range and source maps. 
	Let $X = s^{-1}(\ho)$, and assume the range map $r \vert_X : X \to \go$ is open. If $G$ is exact, then so is $H$.
\end{thm}
\begin{proof}
	As observed in \cite[Example 5.33(7)]{muhly}, $X$ is a $(G, G\vert_\ho)$-equivalence, since $r : X \to \go$ is open. Therefore, $G \vert_\ho$ has a Haar system by 
	\cite[Theorem 2.1]{danaHaar}, and it follows from \cite[Theorem 4.3]{lalondeexact} that $G \vert_\ho$ is exact. Now $H$ is a wide subgroupoid of $G \vert_\ho$ with open
	range and source maps, so Theorem \ref{thm:haar} guarantees that $H$ has a Haar system and Corollary \ref{thm:wideexact} implies that $H$ is exact.
\end{proof}

Now we turn to the question of whether exactness passes to general subgroupoids. That is, if $H$ is a subgroupoid of an exact groupoid $G$, can we still conclude that $H$ is 
exact without assuming it is wide or orbit-wide? Armed with our result for wide subgroupoids, it suffices to know whether reductions of exact groupoids are exact. Specifically, 
we consider reductions to closed invariant subsets of $\go$.

\begin{thm}
\label{thm:invariantexact}
	Let $G$ be a locally compact Hausdorff groupoid with a Haar system, and suppose $F \subseteq \go$ is a closed invariant set. If $G$ is exact, then the reduction 
	$H=G \vert_F$ is exact.
\end{thm}
\begin{proof}
	Since $F$ is invariant, $H$ has a Haar system by \cite[Proposition 5.17]{geoff}. Let $(\A, H, \alpha)$ be a separable groupoid dynamical system, and let $A = \Gamma_0(F, \A)$.
	Since $A$ is a $C_0(F)$-algebra, there is a continuous map $\sigma : \Prim A \to F$. By composing with the inclusion map $F \hookrightarrow \go$, we obtain a continuous
	map $\Prim A \to \go$, so $A$ is also a $C_0(\go)$-algebra. The action of $C_0(\go)$ on $A$ is characterized by
	\[
		(\varphi \cdot a)(u) = \varphi(u) a(u)
	\]
	for $u \in F$. Thus if $u \in F$, we have
	\[
		I_u = C_{0, u}(\go) \cdot A = C_{0, u}(F) \cdot A,
	\] 
	since restriction clearly yields a surjection $C_{0,u}(\go) \to C_{0, u}(F)$. Hence $A/I_u = \A_u$ for $u \in F$. If $u \not\in F$, we claim that $I_u = A$. Given 
	$\varphi \in C_c(F)$, we can use the Tietze extension theorem to find a function $\psi \in C_c(\go)$ with $\psi(u) = 0$ and $\psi(v) = \varphi(v)$ for all 
	$v \in \supp(\varphi)$. Then given $a \in A$, for all $v \in F$ we have
	\[
		(\psi \cdot a)(v) = \psi(v) a(v) = \varphi(v) a(v).
	\]
	Thus $\psi \cdot a = \varphi \cdot a$, and it follows that $C_{0, u}(\go) \cdot A = C_0(F) \cdot A = A$. Therefore, the 
	upper semicontinuous $C^*$-bundle $\B \to \go$ associated to $A$ has fibers
	\[
		\B_u = \begin{cases}
			\A_u & \text{ if } u \in F \\
			0 & \text{ if } u \not \in F.
		\end{cases}
	\]
	Since $F$ is invariant, it is easy to see that $G$ acts on $\B$ via the family $\{\beta_\gamma\}_{\gamma \in G}$, where $\beta_\gamma = 
	\alpha_\gamma$ if $\gamma \in H$, and $\beta_\gamma$ is trivial otherwise. Then
	\[
		\B \rtimes_{\beta, r} G = \A \rtimes_{\alpha, r} H.
	\]
	Suppose $I \subseteq A$ is an $H$-invariant ideal. Then $I$, viewed as an ideal of $B$, is $G$-invariant. If $G$ is assumed to be exact,
	then the sequence
	\[
		0 \to \I \rtimes_{\beta, r} G \to \B \rtimes_{\beta, r} G \to \B/\I \rtimes_{\beta, r} G \to 0
	\]
	is exact. It is straightforward to check that the diagram
	\[
		\xymatrix{
			0 \ar[r] & \I \rtimes_{\beta, r} G \ar[r] \ar[d] & \B \rtimes_{\beta, r} G \ar[r] \ar[d] &  \B/\I \rtimes_{\beta, r} G \ar[r] \ar[d] & 0 \\
			0 \ar[r] & \I \rtimes_{\beta, r} H \ar[r] & \A \rtimes_{\beta, r} H \ar[r] &  \A/\I \rtimes_{\beta, r} H \ar[r] & 0
		}
	\]
	commutes, and the vertical arrows are isomorphisms, so the bottom row is exact. Since $(\A, H, \alpha)$ was an arbitrary dynamical system, it 
	follows that $H$ is exact.
\end{proof}

\begin{cor}
	Let $G$ be an exact groupoid, and $H \subseteq G$ a closed subgroupoid with a Haar system. If $\ho$ is an invariant subset of $\go$, then $H$ is exact.
\end{cor}
\begin{proof}
	By Theorem \ref{thm:invariantexact}, the reduction $G \vert_\ho$ is exact since $\ho$ is a closed invariant set. Then $H$ is a wide subgroupoid of $G \vert_\ho$, hence it is
	exact by Theorem \ref{thm:wideexact}.
\end{proof}

\begin{cor}
\label{cor:exactfibers}
	Let $G$ be an exact group bundle. Then for each $u \in \go$, the fiber $G_u$ is an exact group.
\end{cor}
\begin{proof}
	For each $u \in \go$, the singleton $\{u\}$ is a closed invariant set, and the reduction to $\{u\}$ is precisely $G_u$. Thus the result is immediate from Theorem 
	\ref{thm:invariantexact}.
\end{proof}

In the event that $G$ has continuously varying stabilizers, we know that $\iso(G)$ is an exact group bundle by Corollary \ref{cor:isoexact}. We can now conclude that all of the
isotropy groups must be exact as well.

\begin{cor}
	Let $G$ be an exact groupoid, and assume $G$ has continuously varying stabilizers. Then for each $u \in \go$, the isotropy group $G \vert_u$ is exact.
\end{cor}

If $G$ does not have continuously varying stabilizers, the situation is a little murkier. However, if we assume the orbit space is sufficiently nice, then we can still argue that the
isotropy groups are exact.

\begin{cor}
	Let $G$ be an exact groupoid, and assume that the orbit space $\go/G$ is $T_1$. Then for each $u \in \go$, the isotropy group $G \vert_u$ is exact.
\end{cor}
\begin{proof}
	The requirement that $\go/G$ is $T_1$ guarantees that orbits are closed in $\go$. Let $F = [u]$ be the orbit of $u$. Then $F$ is closed and invariant, so Theorem 
	\ref{thm:invariantexact} guarantees that the reduction $G \vert_F$ is exact. But $G \vert_F$ is a transitive groupoid, which is equivalent to the isotropy group $G \vert_u$. 
	Theorem 4.3 of \cite{lalondeexact} then implies that $G_u$ is exact.
\end{proof}

We can even get a partial result in the other direction. That is, we show that the existence of a certain exact subgroupoid of $G$ guarantees that $G$ is exact. This sort of 
result appears to have been observed in the group case in \cite{claire02}.

\begin{prop}
\label{prop:exactext}
	Let $G$ be a locally compact Hausdorff groupoid, and suppose $H \subseteq G$ is a closed, wide, amenable subgroupoid such that the map $p : G/H \to \go$
	defined by $p([\gamma]) = r(\gamma)$ is proper. Then $G$ is exact.
\end{prop}
\begin{proof}
	Since $p : G/H \to \go$ is proper, $G/H$ is a fiberwise compact $G$-space, in the sense of \cite{claireexact}. Moreover, the transformation groupoid $G \ltimes G/H$ is 
	equivalent to $H$, hence amenable. Thus $G$ acts amenably on a fiberwise compact space, so it is amenable at infinity. It follows from \cite[Proposition 6.7]{claireexact}
	that $G$ is exact.
\end{proof}

\section{A Partial Converse for Transformation Groupoids}
\label{sec:converse}
We can actually extend Proposition \ref{prop:exactext} considerably to obtain a partial converse to Theorem \ref{thm:exacttrans}. In the group case, Proposition \ref{prop:exactext} 
follows from the theorem in Section 7 of \cite{kw99-2}. This section is devoted to proving an analogue of that result for groupoids.

\begin{thm}
\label{thm:exactext2}
	Let $G$ be a locally compact Hausdorff groupoid, and suppose $X$ is a left $G$-space such that the structure map $p : X \to \go$ is proper and the transformation groupoid
	$G \ltimes X$ is exact. Then $G$ is exact.
\end{thm}

The proof relies on the following setup. The results are largely the same as the content of \cite[Lemma 5.6]{claireexact}, though we present the details in a slightly different way. Let 
$(\A, G, \alpha)$ be a separable groupoid dynamical system with $A = \Gamma_0(\go, \A)$ exact. Recall that $C_0(X)$ is a $C_0(\go)$-algebra with fibers given by
\[
	C_0(X)(u) = C_0(p^{-1}(u)) = C(p^{-1}(u)),
\]
since $p^{-1}(u)$ is compact for all $u \in \go$. For brevity, we will let $\Cgh$ denote the associated upper semicontinuous $C^*$-bundle over $\go$. By \cite[Example 4.8]{mw08},
$G$ acts on $\Cgh$ via left translation: for each $\gamma \in G$, we have an isomorphism $\lt_\gamma : \Cgh_{s(\gamma)} \to \Cgh_{r(\gamma)}$ given by
\[
	\lt_\gamma(f)(x) = f(\gamma^{-1} \cdot x).
\]
It now follows from the discussion on page 919 of \cite{KMRW} and \cite[Proposition 6.9]{jonbrown} that $G$ acts on the balanced tensor product $A \otimes_{\go} C_0(X)$. 
More precisely, $G$ acts on the associated upper semicontinuous $C^*$-bundle $\A \otimes_\go \Cgh$, which has fibers
\[
	(\A \otimes_\go \Cgh)_u = \A_u \omax \Cgh_u \cong C(p^{-1}(u), \A_u).
\]
The action is just the tensor product of the original actions. That is, there are isomorphisms $\alpha_\gamma \otimes \lt_\gamma : (\A \otimes_\go \Cgh)_{s(\gamma)} \to 
(\A \otimes_\go \Cgh)_{r(\gamma)}$ for each $\gamma \in G$, which take the form
\[
	(\alpha_\gamma \otimes \lt_\gamma)(f)(x) = \alpha_\gamma \bigl( f(\gamma^{-1} \cdot x) \bigr)
\]
for $f \in C(p^{-1}(s(\gamma)), \A_{s(\gamma)})$. Thus $(\A \otimes_\go \Cgh, G, \alpha \otimes \lt)$ is a groupoid dynamical system.

On the other hand, we can form the pullback bundle $p^*\A \to X$, which admits an action of $G \ltimes X$ via
\[
	\sigma_{(\gamma, x)}(a) = \alpha_{\gamma}(a).
\]
It is straightforward to see that this defines a continuous action. For one, if $(\gamma, \eta) \in \gtwo$, then
\[
	\sigma_{(\gamma, x)(\eta, \gamma^{-1} \cdot x)}(a) = \sigma_{(\gamma \eta, x)}(a) = \alpha_{\gamma \eta} (a) = \alpha_\gamma \bigl( \alpha_\eta(a) \bigr)
		= \sigma_{(\gamma, x)} \bigl( \sigma_{(\eta, \gamma^{-1} \cdot x)}(a) \bigr).
\]
Now suppose $(\gamma_i, x_i) \to (\gamma, x)$ in $G \ltimes X$ and $a_i \to a$ in $p^*\A$ with $a_i \in p^*\A_{s(\gamma_i, x_i)}$ for all $i$. By viewing $p^*\A$
as a subset of $X \times \A$, we can write $(\gamma^{-1}_i \cdot x_i, a_i)$ and $(\gamma^{-1} \cdot x, a)$ in place of $a_i$ and $a$, respectively. It is then clear that $a_i \to a$
in $\A$, so $\alpha_{\gamma_i}(a_i) \to \alpha_{\gamma}(a)$ in $\A$ since $\alpha$ is a continuous action. It follows that $(x_i, \alpha_{\gamma_i}(a_i)) \to (x, \alpha_\gamma(a))$
in $p^*\A$, so $\sigma$ is continuous.

\begin{prop}
\label{prop:c0gh}
	There is an isomorphism 
	\[
		\Phi : p^*\A \rtimes_{\sigma, r} (G \ltimes X) \to (\A \otimes_\go \Cgh) \rtimes_{\alpha \otimes \lt, r} G, 
	\]
	which is characterized by
	\[
		\Phi(f)(\gamma)(x) = f(\gamma, x)
	\]
	for $f \in \Gamma_c(G \ltimes X, r^*(p^*\A))$.
\end{prop}
\begin{proof}
	We really just need to show that our dynamical systems fall under the purview of Theorems \ref{thm:fullcross} and \ref{thm:redcross}. Notice first that the pullback algebra
	can be written as $p^*A = A \otimes_\go C_0(X)$, which is precisely the section algebra of the bundle $\A \otimes_\go \Cgh \to \go$. Viewing $p^*A$ as a $C_0(\go)$-algebra,
	the fibers are precisely
	\[
		p^*\A_{u} = \Gamma_0(p^{-1}(u), p^*\A) = C(p^{-1}(u), \A_u) = (\A \otimes_\go \Cgh)_u.
	\]
	Thus the bundle $\A \otimes_\go \Cgh \to \go$ is exactly the bundle $\B$ from Theorem \ref{thm:fullcross}. The action $\alpha \otimes \lt$ is exactly what we expect, too:
	for $f \in C(p^{-1}(s(\gamma)), \A_{s(\gamma)})$,
	\[
		(\alpha \otimes \lt)_\gamma(f)(x) = \alpha_\gamma \bigl( f(\gamma^{-1} \cdot x) \bigr) = \sigma_{(\gamma, x)} \bigl( f(\gamma^{-1} \cdot x) \bigr).
	\]
	Therefore, the existence of the desired isomorphism follows as a special case of Theorem \ref{thm:redcross}.
\end{proof}

Notice that if $A$ is exact, then $p^*A = A \otimes_\go C_0(X)$ is exact. Therefore, if the transformation groupoid $G \ltimes X$ is also assumed to be exact, then the 
reduced crossed product $p^*\A \rtimes_{\sigma, r} (G \ltimes X)$ is exact by \cite[Theorem 6.14]{lalonde2014}. It then follows from the isomorphism of Proposition 
\ref{prop:c0gh} that $(\A \otimes_\go \Cgh) \rtimes_{\alpha \otimes \lt, r} G$ is an exact $C^*$-algebra.

The next step is to show that $\A \rtimes_{\alpha, r} G$ embeds into $(\A \otimes_\go \Cgh) \rtimes_{\alpha \otimes \lt, r} G$. Since the fibers of $p : X \to \go$ are compact,
each fiber $\Cgh_u = C(p^{-1}(u))$ is unital. Thus we have fiberwise embeddings $\A_u \hookrightarrow \A_u \omax \Cgh_u$ given by $a \mapsto a \otimes 1$. Indeed, we claim that
these homomorphisms yield a continuous $C^*$-bundle homomorphism $\hat{\iota} : \A \hookrightarrow \A \otimes_\go \Cgh$. To prove it, we will show that there
is a $C_0(\go)$-linear embedding $\iota : A \to A \otimes_\go C_0(X)$.

\begin{prop}
	There is an injective $C_0(\go)$-linear homomorphism $\iota : A \to A \otimes_\go C_0(X)$ characterized by
	\[
		\iota(a)(u) = a(u) \otimes 1
	\]
	for all $u \in \go$. Furthermore, $\iota$ is $G$-equivariant.
\end{prop}
\begin{proof}
	To construct $\iota$, we will first show that there is an embedding of $A$ into the pullback algebra $p^*A = \Gamma_0(X, p^*\A)$ and then compose with the natural
	isomorphism $p^*A \cong A \otimes_\go C_0(X)$. We need to appeal to \cite[Proposition 3.4.4]{lalondethesis}. Given $a \in A$, define a map $\tilde{f} : X \to \A$ by
	\[
		\tilde{f}(x) = a(p(x)).
	\]
	Then $\tilde{f}$ is clearly continuous: if $x_i \to x$ in $X$, then $p(x_i) \to p(x)$ in $\go$, so $a(p(x_i)) \to a(p(x))$ since $a \in \Gamma_0(\go, \A)$. Moreover, it is 
	straightforward to show that $\tilde{f}$ vanishes at infinity. Let $\varepsilon > 0$ be given, and let $K = \{ u \in \go : \norm{a(u)} \geq \varepsilon \}$. By definition we have
	$\| \tilde{f}(x) \| = \norm{a(p(x))}$ for all $x \in X$, so
	\[
		\{ x \in X : \| \tilde{f}(x) \| \geq \varepsilon \} = p^{-1}(K)
	\]
	is compact, since we have assumed that $p : X \to \go$ is a proper map. Therefore, the section $p^*a$ of $p^*\A \to X$ defined by
	\begin{equation}
	\label{eq:pstar}
		p^*a(x) = (x, \tilde{f}(x)) = (x, a(p(x)))
	\end{equation}
	belongs to $\Gamma_0(X, p^*\A)$ by \cite[Proposition 3.4.4]{lalondethesis}.
	
	Now we check that the map $a \mapsto p^*a$ is an injective homomorphism. Well, it is immediate from \eqref{eq:pstar} that the map is a homomorphism. If $p^*a = 0$ for some
	$a \in A$, then we have $a(p(x)) = 0$ for all $x \in X$. Since $p$ is surjective, this means $a = 0$, and the map is injective. If we let $\Theta : p^*A \to A \otimes_\go C_0(X)$
	denote the natural isomorphism afforded by Proposition \ref{prop:c0fiber}, then we have
	\[
		\Theta(p^*a)(u)(x) = p^*a(x) = a(p(x)) = a(u)
	\]
	for all $u \in \go$ and $x \in p^{-1}(u)$. That is, $\Theta(p^*a)(u)$ (viewed as an element of $C(p^{-1}(u), \A_u)$) is the constant function $x \mapsto a(u)$, which we identify
	with the elementary tensor $a(u) \otimes 1$. Thus our embedding $\iota : A \to A \otimes_\go C_0(X)$ takes the desired form:
	\[
		\iota(a)(u) = \Theta(p^*a)(u) = a(u) \otimes 1
	\]
	for all $u \in \go$.
	
	It remains to see that $\iota$ is $C_0(\go)$-linear. Let $a \in A$ and $\varphi \in C_0(\go)$. Then for all $u \in \go$, we have
	\begin{align*}
		\iota(\varphi \cdot a)(u) &= (\varphi \cdot a)(u) \otimes 1 \\
			&= \varphi(u) a(u) \otimes 1 \\
			&= \varphi(u) \bigl( a(u) \otimes 1 \bigr) \\
			&= \varphi(u) \iota(a)(u) \\
			&= \bigl( \varphi \cdot \iota(a) \bigr)(u),
	\end{align*}
	so $\iota(\varphi \cdot a) = \varphi \cdot \iota(a)$. Thus $\iota$ is $C_0(\go)$-linear. Consequently, it induces fiberwise homomorphisms $\iota_u : \A_u \to \A_u \omax \Cgh_u$,
	which are given by $\iota_u(a) = a \otimes 1$.
	
	Finally, we check that $\iota$ is $G$-equivariant. This amounts to verifying that for all $a \in A$ and all $\gamma \in G$,
	\[
		\iota_{r(\gamma)} \bigl( \alpha_\gamma(a) \bigr) = (\alpha \otimes \lt)_\gamma \bigl( \iota_{s(\gamma)}(a) \bigr).
	\]
	On the left hand side we have
	\[
		\iota_{r(\gamma)} \bigl(\alpha_\gamma(a) \bigr) = \alpha_\gamma(a) \otimes 1,
	\]
	while on the right,
	\begin{align*}
		(\alpha \otimes \lt)_\gamma \bigl( \iota_{s(\gamma)}(a) \bigr) &= (\alpha \otimes \lt)_\gamma \bigl( a \otimes 1 \bigr) \\
			&= \alpha_\gamma (a) \otimes \lt_\gamma(1) \\
			&= \alpha_\gamma(a) \otimes 1 \\
			&= \iota_{r(\gamma)} \bigl( \alpha_\gamma(a) \bigr).
	\end{align*}
	Thus $\iota$ is $G$-equivariant.
\end{proof}

Since $\iota : A \to A \otimes_\go C_0(X)$ is $G$-equivariant, \cite[Proposition 6.3]{lalonde2014} guarantees that there is a homomorphism
$\iota \rtimes \id : \Gamma_c(G, r^*\A) \to \Gamma_c(G, r^*(\A \otimes_\go \Cgh))$ characterized by
\[
	\iota \rtimes \id(f)(\gamma) = \iota_{r(\gamma)}(f(\gamma)) = f(\gamma) \otimes 1.
\]
Moreover, $\iota \rtimes \id$ extends to a homomorphism between the associated full crossed products. Proposition 6.10 of \cite{lalonde2014} then implies that there is also 
a homomorphism $\iota \rtimes \id : \A \rtimes_{\alpha, r} G \to (\A \otimes_\go \Cgh) \rtimes_{\alpha \otimes \id, r} G$ at the level of the reduced crossed products. All that 
remains is to prove that $\iota \rtimes \id$ is injective.

\begin{prop}
\label{prop:embed}
	There is an injective homomorphism  
	\[
		\iota \rtimes \id : \A \rtimes_{\alpha, r} G \to (\A \otimes_\go \Cgh) \rtimes_{\alpha \otimes \lt, r} G,
	\]
	which is characterized by
	\[
		\iota \rtimes \id(f)(\gamma) = f(\gamma) \otimes 1,
	\]
	for $f \in \Gamma_c(G, r^*\A)$.
\end{prop}
\begin{proof}
	We first need to check that $\iota \rtimes \id : \Gamma_c(G, r^*\A) \to \Gamma_c(G, r^*(\A \otimes_\go \Cgh))$ is injective. Suppose $\iota \rtimes \id(f) = 0$ for some 
	$f \in \Gamma_c(G, r^*\A)$. Then for all $\gamma \in G$,
	\[
		\iota \rtimes \id(f)(\gamma) = \iota_{r(\gamma)}(f(\gamma)) = 0, 
	\]
	and since $\iota_{r(\gamma)}$ is injective, $f(\gamma) = 0$. Thus $f = 0$, and $\iota \rtimes \id$ is injective on $\Gamma_c(G, r^*\A)$. Now the argument following the
	proof of \cite[Proposition 6.10]{lalonde2014} shows that
	\[
		\norm{\iota \rtimes \id(f)}_r = \norm{f}_r
	\]
	for all $f \in \Gamma_c(G, r^*\A)$. Hence $\iota \rtimes \id$ is isometric, so it extends to an injective homomorphism $\iota \rtimes \id : \A \rtimes_{\alpha, r} G \to 
	(\A \otimes_\go \Cgh) \rtimes_{\alpha \otimes \lt, r} G$. 
\end{proof}

\begin{proof}[Proof of Theorem \ref{thm:exactext2}]
	Let $(\A, G, \alpha)$ be a separable groupoid dynamical system and let $I \subseteq A$ be a $G$-invariant ideal. If we form the pullback dynamical system 
	$(p^*\A, G \ltimes X, \sigma)$, then the pullback $C^*$-algebra $p^*I$ is a $(G \ltimes X)$-invariant ideal of $p^*A$ by \cite[Proposition 3.2]{lalondeexact}. Therefore, we 
	obtain a short exact sequence
	\[
		0 \to p^*\I \rtimes_{\sigma, r} (G \ltimes X) \to p^*\A \rtimes_{\sigma, r} (G \ltimes X) \to p^*(\A/\I) \rtimes_{\sigma, r} (G \ltimes X) \to 0
	\]
	since $G \ltimes X$ is exact. By Proposition \ref{prop:embed}, we have embeddings of $\I \rtimes_{\alpha, r} G$, $\A \rtimes_{\alpha, r} G$, and $\A/\I \rtimes_{\alpha, r} G$
	into the respective pullback crossed products. This yields a diagram
	\[
		\xymatrix@C=12pt{
			0 \ar[r] & \I \rtimes_{\alpha, r} G \ar[r] \ar[d] & \A \rtimes_{\alpha, r} G \ar[r] \ar[d] & \A/\I \rtimes_{\alpha, r} G \ar[r] \ar[d] & 0 \\
			0 \ar[r] & p^*\I \rtimes_{\sigma, r} (G \ltimes X) \ar[r] & p^*\A \rtimes_{\sigma, r} (G \ltimes X) \ar[r] & p^*(\A/\I) \rtimes_{\sigma, r} (G \ltimes X) \ar[r] & 0
		}
	\]
	which is easily seen to commute. If $a$ belongs to the kernel of the map $\A \rtimes_{\alpha, r} G \to \A/\I \rtimes_{\alpha, r} G$, then $\iota \rtimes \id(a)$ is in the kernel
	of $p^*\A \rtimes_{\sigma, r} (G \ltimes X) \to p^*(\A/\I) \rtimes_{\sigma, r} (G \ltimes X)$. Hence
	\[
		\iota \rtimes \id (a) \in \iota \rtimes \id( \A \rtimes_{\alpha, r} G ) \cap p^*\I \rtimes_{\sigma, r} (G \ltimes X).
	\]
	We claim that $\iota \rtimes \id(a) \in \iota \rtimes \id(\I \rtimes_{\alpha, r} G)$. It will suffice to show that 
	\[
		\iota \rtimes \id(\A \rtimes_{\alpha, r} G) \cap (\I \otimes_\go \Cgh) \rtimes_{\alpha \otimes \lt, r} G = \iota \rtimes \id(\I \rtimes_{\alpha, r} G).
	\]
	Suppose first that $f \in \Gamma_c(G, r^*\A)$ and $\iota \rtimes \id(f) \in (\I \otimes_\go \Cgh) \rtimes_{\alpha \otimes \lt, r} G$. Then we have 
	\[
		\iota \rtimes \id(f)(\gamma) = f(\gamma) \otimes 1 \in \I_{r(\gamma)} \omax \Cgh_{r(\gamma)}
	\]
	for all $\gamma \in G$, meaning that $f(\gamma) \in \I_{r(\gamma)}$ for all $\gamma \in G$. That is, $f \in \Gamma_c(G, r^*\I)$. It follows that 
	\[
		\iota \rtimes \id(\A \rtimes_{\alpha, r} G) \cap (\I \otimes_\go \Cgh) \rtimes_{\alpha \otimes \lt, r} G \subseteq \iota \rtimes \id(\I \rtimes_{\alpha, r} G).
	\]
	The other containment is clear, so we have the desired equality. Therefore, $a \in \I \rtimes_{\alpha, r} G$, and the top row of the diagram above is exact. Thus $G$ is exact.
\end{proof}

\begin{cor}
\label{cor:exactsub}
	Let $G$ be a locally compact Hausdorff groupoid, and suppose $H \subseteq G$ is a closed, wide, exact subgroupoid such that the map 
	$p : G/H \to \go$ is proper. Then $G$ is exact.
\end{cor}
\begin{proof}
	Since $H$ is exact, the transformation groupoid $G \ltimes G/H$ is exact. Since the map $p : G/H \to \go$ is proper, it follows from Theorem \ref{thm:exactext2} that 
	$G$ is exact.
\end{proof}

\begin{rem}
By taking $H = G$, we can see that the converse of Corollary \ref{cor:exactsub} holds as well. In this case we have $G/H \cong \go$, so the map $p : G/H \to \go$ is certainly proper,
and $G \ltimes G/H \cong G$. Thus $G$ is exact if and only if $G \ltimes G/H$ is.
\end{rem}

The setup of Theorem \ref{thm:exactext2} is reminiscent of the notion of amenability at infinity for groupoids. Recall that a groupoid $G$ is \emph{amenable
at infinity} if there is a $G$-space $X$ such that $p: X \to \go$ is proper and $G \ltimes X$ is an amenable groupoid. It is known that if $G$ is amenable at infinity, then $G$
is exact \cite[Proposition 6.7]{claireexact}. The converse is known to hold if $G$ is assumed to be weakly inner amenable, as defined by Anantharaman-Delaroche in 
\cite[Definition 4.2]{claireexact}. It is unknown whether exactness implies amenability at infinity in general. (The two properties are now known to be equivalent for groups by
\cite{bcl}, however.) Thanks to Theorem \ref{thm:exactext2}, we have another way of phrasing this open question.

\begin{ques}
	If a groupoid $G$ acts on a fiberwise compact space $X \to \go$ such that $G \ltimes X$ is an exact groupoid, must $G$ act amenably on some fiberwise compact space?
\end{ques}

\section{Inner Exactness}
\label{sec:inner}

We now conclude with a brief discussion inner exactness, as introduced by Anantharaman-Delaroche in \cite{claireweak}. A locally compact groupoid $G$ is said to be
\emph{inner exact} if the sequence
\[
	0 \to C_r^*(G \vert_U) \to C_r^*(G) \to C_r^*(G \vert_{\go \backslash U}) \to 0
\]
is exact for any open invariant set $U \subseteq \go$. It is clear that any exact groupoid is inner exact. However, this condition is strictly weaker than exactness, since any 
locally compact group is automatically inner exact. The importance of inner exactness became clear long before it was formally defined in \cite{claireweak}, since many of the 
known examples of non-exact groupoids actually fail to be inner exact. In particular, the non-exact HLS groupoids constructed by Higson, Lafforgue, and Skandalis in \cite{hls} 
as counterexamples to the Baum-Connes conjecture are not inner exact.

It is natural to ask which of the known permanence properties for exact groupoids translate over to inner exact groupoids. We begin our discussion with one such result, which is a 
version of \cite[Theorem 4.8]{lalondeexact} for inner exact groupoids.

\begin{thm}
\label{thm:equivinner}
	Let $G$ and $H$ be groupoids, and suppose $Z$ is a $(G,H)$-equivalence. If $G$ is inner exact, then so is $H$.
\end{thm}
\begin{proof}
	Let $s_Z : Z \to \ho$ and $r_Z : Z \to \go$ denote the structure maps for the $G$- and $H$-actions on $Z$, respectively, and suppose $U \subseteq \ho$ is an open, 
	$H$-invariant set. Then $s_Z^{-1}(U)$ is open in $Z$, and since the range map $r_Z : Z \to \go$ is open, $V = r_Z(s_Z^{-1}(U))$ is open in $\go$.
	We claim that $V$ is also $G$-invariant. Suppose $u \in V$ and $\gamma \in G$ with $s(\gamma) = u$. Choose $z \in s_Z^{-1}(U)$ with $r_Z(z) = u$. Then
	\[
		s_Z(\gamma \cdot z) = s_Z(z) \in U,
	\]
	so $\gamma \cdot z \in s_Z^{-1}(U)$. This implies that $r_Z(\gamma \cdot z) = r(\gamma) \in r_Z(s_Z^{-1}(U)) = V$, so $V$ is $G$-invariant.
	
	Now let $L = G \sqcup Z \sqcup \Zop \sqcup H$ denote the linking groupoid for the $(G,H)$-equivalence $Z$, as defined in \cite[Lemma 2.1]{sims-williams2012}. We claim 
	that $U \cup V$ is an $L$-invariant subset of $\lo$. Clearly we just need to consider the elements of $L$ coming from $Z$ and $\Zop$. Suppose first that $z \in Z$ with 
	$s(z) \in U \cup V$. Then $s_Z(z) \in U$, so $z \in s_Z^{-1}(U)$, and we have $r_Z(z) \in V$ by construction. A similar argument works for elements of $\Zop$, so 
	$U \cup V$ is invariant in $\lo$.
	
	Since $U \cup V$ is an open, invariant subset of $\lo$, we can form the sequence of reduced $C^*$-algebras
	\begin{equation}
	\label{eq:lsequence}
		0 \to C_r^*(L \vert_{U \cup V}) \to C_r^*(L) \to C_r^*(L \vert_{(U \cup V)^c}) \to 0.
	\end{equation}
	We claim that this sequence is exact if and only if 
	\begin{equation}
	\label{eq:gsequence}
		0 \to C_r^*(G \vert_V) \to C_r^*(G) \to C_r^*(G \vert_{V^c}) \to 0
	\end{equation}
	is exact, or similarly if and only if
	\begin{equation}
	\label{eq:hsequence}
		0 \to C_r^*(H \vert_U) \to C_r^*(H) \to C_r^*(H \vert_{U^c}) \to 0
	\end{equation}
	is exact. To prove it, we consider the dynamical system $(C_0(\lo), L, \lt)$. The ideal $C_0(U \cup V)$ of $C_0(\lo)$ is $L$-invariant, and it is easy to see that
	\[
		C_0(U \cup V) \cap C_0(\go) = C_0(V),
	\]
	which is a $G$-invariant ideal in $C_0(\go)$. Likewise,
	\[
		C_0(U \cup V) \cap C_0(\ho) = C_0(U)
	\]
	is an $H$-invariant ideal in $C_0(\ho)$. It then follows from \cite[Corollary 4.6]{lalondeexact} that the sequences \eqref{eq:gsequence} and \eqref{eq:hsequence} are exact 
	if and only if \eqref{eq:lsequence} is exact. Thus if $G$ is assumed to be inner exact, then $H$ must be as well.
\end{proof}

We can now use Theorem \ref{thm:equivinner} to establish an analogue of Corollary \ref{cor:exactsub} for inner exactness. 

\begin{thm}
	Let $G$ be a locally compact Hausdorff groupoid, and suppose $H$ is a closed, wide, inner exact subgroupoid such that the map 
	$p: G/H \to \go$ is proper. Then $G$ is inner exact.
\end{thm}
\begin{proof}
	Since $H$ is inner exact, the transformation groupoid $G \ltimes G/H$ is inner exact by Theorem \ref{thm:equivinner}. It is easily checked that Proposition \ref{prop:embed} 
	guarantees that there is an embedding $C_r^*(G) \hookrightarrow C_r^*(G \ltimes G/H)$ since $p : G/H \to \go$ is proper. Suppose $U \subseteq \go$ is open and invariant, 
	and set $V = p^{-1}(U)$. Then $V$ is open in $G/H$, and if $[x] \in V$ and $\gamma \in G$ with $s(\gamma) = p([x]) = r(x)$, then
	\[
		p(\gamma \cdot [x]) = p([\gamma x]) = r(\gamma x) = r(\gamma) \in U
	\]
	since $U$ is invariant. Thus $\gamma \cdot [x] \in V$, so $V$ is $G$-invariant. Consequently, the sequence
	\[
		0 \to C_r^*((G \ltimes G/H) \vert_V) \to C_r^*(G \ltimes G/H) \to C_r^*((G \ltimes G/H) \vert_{V^c}) \to 0
	\]
	is exact since $G \ltimes G/H$ is inner exact. We then get a commuting diagram like the one in the proof of Theorem \ref{thm:exactext2},
	\[
		\xymatrix@C=12pt{
			0 \ar[r] & C_r^*(G \vert_U) \ar[r] \ar[d] & C_r^*(G) \ar[r] \ar[d] & C_r^*(G \vert_{U^c}) \ar[r] \ar[d] & 0 \\
			0 \ar[r] & C_r^*((G \ltimes G/H) \vert_V) \ar[r] & C_r^*(G \ltimes G/H) \ar[r] & C_r^*((G \ltimes G/H) \vert_{V^c}) \ar[r] & 0
		}
	\]
	where the vertical arrows are injective. The result then follows from the same diagram-chasing argument as in the proof of Theorem \ref{thm:exactext2}.
\end{proof}

Next we investigate a negative result on quotients of inner exact groupoids. As Ozawa has observed in \cite{ozawa}, quotients of exact groups need not be exact. Thus the 
question of whether exactness descends to quotients of groupoids is already settled. In the next example, we observe that the situation can actually be much worse: a quotient of 
an exact groupoid by a closed, normal subgroupoid need not be inner exact. Naturally enough, Willett's example of a non-amenable groupoid with weak containment lies at the 
heart of the discussion.

\begin{exmp}
	Let $G = \mathbb{F}_2 \times \hat{\mathbb{N}}$, where $\mathbb{F}_2$ denotes the free group on two generators and $\hat{\mathbb{N}}$ is the one-point compactification of
	the natural numbers. We view $G$ as a group bundle over $\hat{\mathbb{N}}$. Notice that $G$ is the transformation groupoid associated to the trivial action of the exact group
	$\mathbb{F}_2$ on $\hat{\mathbb{N}}$, so it is exact by Theorem \ref{thm:exacttrans}. Now let $\{K_n\}_{n=1}^\infty$ be the sequence of subgroups of $\mathbb{F}_2$ 
	described in \cite[Lemma 2.8]{willett}, which are defined to be
	\[
		K_n = \bigcap \ker \varphi,
	\]
	where $\varphi$ ranges over all group homomorphisms $\varphi : \mathbb{F}_2 \to \Gamma$ with $\abs{\Gamma} \leq n$. Then each $K_n$ is a finite index, normal 
	subgroup of $\mathbb{F}_2$, the subgroups are nested, and $\bigcap_{n=1}^\infty K_n$ is trivial. If we let $K_\infty$ be the trivial subgroup, then the bundle
	\[
		H = \bigcup_{n \in \hat{\mathbb{N}}} K_n
	\]
	is a clopen, wide, normal subgroupoid of $G$. Hence $H$ is exact by Theorem \ref{thm:wideexact}. However, it is easy to see that $H$ encodes the equivalence relation
	on $G$ described in \cite[Section 2]{hls}, so the quotient $G/H$ is precisely the HLS groupoid constructed by Willett in \cite{willett}. But $G/H$ is not inner exact by construction.
	
	Note that it is not necessary to work exclusively with Willett's groupoid here. Any non-inner exact HLS groupoid (such as the ones described in Section 2 of \cite{hls} and 
	Section 8.4 of \cite{claireexact}) will do.
\end{exmp}

The situation for transformation groupoids and subgroupoids is murkier. Thanks to an example of Baum, Guentner, and Willett from \cite{bgw}, we can see that the transformation 
groupoid associated to a \emph{group} action might fail to be inner exact. Given a Gromov monster group $\Gamma$, the authors construct an exact sequence of separable 
commutative $\Gamma$-$C^*$-algebras
\[
	0 \to C_0(\Gamma) \to C_0(Z) \to C_0(\partial Z) \to 0	
\]
such that the sequence of reduced crossed products
\begin{equation}
\label{eq:bgw}
	0 \to C_0(\Gamma) \rtimes_r \Gamma \to C_0(Z) \rtimes_r \Gamma \to C_0(\partial Z) \rtimes_r \Gamma \to 0
\end{equation}
is not exact \cite[Theorem 7.4]{bgw}. As with the HLS examples in \cite{hls}, the failure of exactness is realized at the level of $K$-theory. If we define $G = \Gamma \ltimes Z$, then
\eqref{eq:bgw} is isomorphic to the sequence
\[
	0 \to C_r^*(G \vert_{\Gamma}) \to C_r^*(G) \to C_r^*(G \vert_{\partial Z}) \to 0,
\]
and it follows that $G$ is not inner exact.

The above discussion shows that the analogue of Theorem \ref{thm:exacttrans} for inner exactness does not hold. This suggests that one would need a different approach to show 
that exactness descends to wide subgroupoids. It may be possible to simply weaken the hypotheses of Theorem \ref{thm:exacttrans}---it would clearly suffice to show that if $G$
is inner exact and $H \subseteq G$ is a wide subgroupoid, then $G \ltimes G/H$ is inner exact. In any event, the example described above seems to cast a pall over the whole
situation, and we simply pose the following open question.

\begin{ques}
	If $G$ is inner exact and $H \subseteq G$ is a wide subgroupoid, must $H$ also be inner exact?
\end{ques}

We close with another question that may be of interest to others, which is inspired in part by the aforementioned examples of HLS groupoids. Suppose $G$ is a group bundle.
For $G$ to be exact, it is necessary that each fiber is an exact group by Corollary \ref{cor:exactfibers}. This condition is not sufficient, as the non-exact HLS examples show. 
However, those examples fail to be exact precisely because they are not inner exact. Perhaps these two conditions together guarantee exactness.

\begin{ques}
	Let $G$ be a group bundle. If $G$ is inner exact and $G_u$ is an exact group for all $u \in \go$, must $G$ be exact?
\end{ques}

\section*{Acknowledgements}
The author would like to thank Jean Renault for useful suggestions, and Dana Williams for providing helpful comments during 
the preparation of this paper. 

\bibliographystyle{amsplain}
\bibliography{/Users/slalonde/Documents/Research/SharedFiles/GroupBib.bib}
\end{document}